\documentclass[reqno]{amsart}
\usepackage{amsmath, amsthm, amssymb, amstext, mathrsfs}
\usepackage{graphicx}
\usepackage{bm}

\usepackage[left=3cm,right=3cm,top=2cm,bottom=2cm,includeheadfoot]{geometry}
\usepackage{hyperref,xcolor}
\hypersetup{
	pdfborder={0 0 0},
	colorlinks,
}

\usepackage[shortlabels]{enumitem}
\usepackage{amsfonts}
\usepackage[utf8]{inputenc}

\DeclareMathOperator*{\esssup}{ess\,sup}
\DeclareMathOperator*{\essinf}{ess\,inf}
\DeclareMathOperator*{\loc}{loc}
\DeclareMathOperator*{\supp}{supp}

\newcommand*\diff{\mathop{}\!\mathrm{d}}

\renewcommand{\l}{\left}
\renewcommand{\r}{\right}

\newcommand{\round}[1]{\left(#1\right)}

\newcommand{\scal}[1]{\left\langle#1\right\rangle}

\newcommand{\wt}{\widetilde}
\newcommand{\wh}[1]{\widehat{#1}}

\newcommand{\R}{{\mathbb R}}
\newcommand{\N}{{\mathbb N}}
\newcommand{\RN}{\mathbb{R}^N}
\newcommand{\Lp}[1]{L^{#1}(\Omega)}

\newcommand{\close}{\overline{\Omega}}

\newcommand{\eps}{\varepsilon}
\newcommand{\de}{\delta}

\newcommand{\into}{\int_{\Omega}}

\newcommand{\Assg}[1]{\textup{(g)}}

\begin{document}
	\title[Global boundedness for generalized Schr\"{o}dinger-type double phase problems in $\R^N$]
	{Global boundedness for Generalized Schr\"odinger-Type Double Phase Problems  in  $\R^N$ and Applications to Supercritical Double Phase Problems}

	\author[H.H. Ha]{Hoang Hai Ha}\address{Hoang Hai Ha \newline
		Department of Mathematics, Faculty of Applied Science , Ho Chi Minh City University of Technology (HCMUT), 268 Ly Thuong Kiet Street, District 10, Ho Chi Minh City, Vietnam \newline
		Vietnam National University Ho Chi Minh City, Linh Trung Ward, Thu Duc City,  Vietnam}
	\email{hoanghaiha@hcmut.edu.vn}
	
	\author[K. Ho]{Ky Ho}
	\address{Ky Ho\newline
		Department of Mathematics and Statistics, University of Economics Ho Chi Minh City, 59C, Nguyen Dinh Chieu Street, Ho Chi Minh City, Vietnam}
	\email{kyhn@ueh.edu.vn}
	
	\author[B.T. Quan]{Bui The Quan}
	\address{Bui The Quan\newline
		Department of Mathematics, Ho Chi Minh City University of Education, Ho Chi Minh City, Vietnam}
	\email{quanbt@hcmue.edu.vn}   
	
	\author[I. Sim]{Inbo Sim}
	\address{Inbo Sim \newline
		Department of Mathematics, University of Ulsan, Ulsan 44610, Republic of Korea}
	\email{ibsim@ulsan.ac.kr}
	

	\begin{abstract}
	We establish two global boundedness results for weak solutions to generalized Schr\"{o}dinger-type double phase problems with variable exponents in $\R^N$ under new critical growth conditions optimally introduced in \cite{HH25,HW2022}. More precisely, for the case of subcritical growth, we employ the De Giorgi iteration with a suitable localization method in $\R^N$ to obtain \textit{a priori} bounds. As a byproduct, we derive the decay property of weak solutions.  For the case of critical growth,  using the De Giorgi iteration with a localization adapted to the critical growth, we prove the global boundedness. As an interesting application of these results, the existence of weak solutions for supercritical double phase problems is shown. These results are new even for problems with constant exponents in $\R^N$.
	\end{abstract}
	
	\maketitle \numberwithin{equation}{section}
	\newtheorem{theorem}{Theorem}[section]
	\newtheorem{lemma}[theorem]{Lemma}
	\newtheorem{definition}[theorem]{Definition}
	\newtheorem{claim}[theorem]{Claim}
	\newtheorem{proposition}[theorem]{Proposition}
	\newtheorem{remark}[theorem]{Remark}
	\newtheorem{corollary}[theorem]{Corollary}
	\newtheorem{example}[theorem]{Example}
	\newtheorem{assumption}[theorem]{Assumption}
	\allowdisplaybreaks

\section{Introduction}\label{Intro}

Since Hilbert’s 19th problem, the study of regularity in partial differential equations (PDEs) has become a central theme in the analysis of differential equations. A widely used approach to establishing regularity for a given PDE is to construct an appropriate function space for solutions, often incorporating Sobolev-type embeddings, and to apply De Giorgi–Nash–Moser iteration schemes.

\vskip5pt
In this article, we are concerned with generalized Schr\"{o}dinger-type double phase problems with variable exponents in  $\R^N$ $(N\geq 2)$, which have attracted considerable attention not only among physicists but also among mathematicians, of the form:
\begin{eqnarray}\label{eq1}
	-\operatorname{div}\mathcal{A}(x,u,\nabla u)+ V(x)\Big(|u|^{p(x)-2} + a(x) |u|^{q(x)-2}\Big)u =\mathcal{B}(x,u,\nabla u) \quad \text{in } \R^N,
\end{eqnarray} 
where $a,p,q \in C^{0,1}(\RN)$  such that $a(x) \geq 0$ and $1<p(x)<q(x)<N$ for all $x\in\RN$, and the potential $V\in L^1_{\loc}\left(\R^N\right)$ with $V(x) \geq 0$ for a.a. $x\in \mathbb{R}^N$. The operator $\operatorname{div}\mathcal{A}(x,u,\nabla u)$ generalizes the double phase operator
\begin{equation}\label{p(.)-La}
\operatorname{div}\left(|\nabla u|^{p(x)-2}\nabla u+ a(x) |\nabla u|^{q(x)-2}\nabla u\right),
\end{equation}
and  $\mathcal{B}:\R^N\times \R \times \R^N \to \R$ is a Carath\'eodory function. Some further conditions will be presented later.

\vskip5pt
The origin of double phase problems dates back to the 1980s. In an effort to model  strongly anisotropic materials, Zhikov \cite{Zhikov-1986} initiated the study of a functional 
\begin{equation}\label{I}
	I(u) = \int_\Omega f\left(x, \nabla u \right) \diff x,
\end{equation}
where $\Omega\subset\mathbb{R}^N$ is a bounded domain and the integrand $f:\Omega\times\mathbb{R}^N\to \mathbb{R}$ satisfies
\begin{equation}\label{cond}
	c_1 |\xi|^p \le f(x,\xi) \le c_2 \left(|\xi|^q+1\right), \quad 1<p<q<\infty.
\end{equation}
A typical example is the double phase operator
\begin{equation}\label{db}
I(u) = \int_\Omega \left(|\nabla u|^{p} + a(x) |\nabla u|^{q} \right) \diff x,
\end{equation}
where the integrand exhibits a switch between two distinct elliptic behaviors.  The functional \eqref{db} is relevant in the context of homogenization and elasticity, where the function $a(\cdot)$ governs the geometry of a composite material composed of two constituents  with hardening exponents $p$ and $q$; see e.g., \cite{Zhikov-1986, Zhikov-1995, Zhikov1997, ZKO}.  Double phase problems also arise in various physical applications \cite{Ba-Ra-Re, Ben2000, Che2005}. Single $p(\cdot)$-Laplacian, namely the operator given by \eqref{p(.)-La} with $a(\cdot)\equiv 0$, can be found in many applications, e.g., image processing \cite{CLR05} and electrorheological fluids \cite{Ru2000}. 

\vskip5pt
From the  view of regularity, Marcellini \cite{Mar89b, Mar91}  was the first to investigate the H\"older continuity of minimizers of the functional $I$ given in \eqref{I}, as well as the weak solutions to the corresponding Euler equation under the nonstandard growth condition \eqref{cond}. In this direction, we refer the reader to the outstanding works focused on double phase operators—among others too numerous to list—such as \cite{BCM15,BCM16,BCM18,BKM15,BO20,CM15a,CM15b,DeF18,ELM,HHT17,Mar89b,Mar91,Ok18,Ok20,RT20}, and the references therein. 


\vskip5pt
It is worth noting that although a large number of papers have addressed the regularity theory for double phase problems,  most of them primarily focus on the regularity of minimizers of double phase functionals. 
To the best of our knowledge, the only result concerning the Hölder continuity of weak solutions for problems involving the double phase operator given in \eqref{p(.)-La} on a bounded domain is found in \cite{RK25}. For such problems, a natural solution space is the Musielak-Orlicz Sobolev space $W^{1,\mathcal{H}}(\Omega)$, which was studied in \cite{CS16, crespo2022new}, inspired by the foundational works  \cite{fan2012,Mus} (see Section~\ref{Pre}), where $\mathcal{H} : \overline{\Omega} \times [0,\infty) \to [0,\infty)$ is defined by
$$
\mathcal{H}(x,t) = t^{p(x)} +a(x) t^{q(x)}.
$$

\vskip5pt
To optimize the class of power-law growth, Ho--Winkert \cite{HW2022} recently investigated the following critical embedding with bounded underlying  domain
\begin{equation}\label{CE}
W^{1,\mathcal{H}}(\Omega) \hookrightarrow L^{\mathcal{G}^*}(\Omega),
\end{equation}
where $\mathcal{G}^* : \overline{\Omega} \times [0,\infty) \to [0,\infty)$ is defined by
$$
\mathcal{G}^* (x,t) := t^{p^*(x)} +a(x)^{\frac{q^*(x)}{q(x)}} t^{q^*(x)}
$$
with $r^*(x):=\frac{N r(x)}{N-r(x)}$ for a function $r \in C(\overline{\Omega})$ satisfying $1<r(x)<N$ for all $x\in\overline{\Omega}$. The exponent $\frac{q^*(\cdot)}{q(\cdot)}$ is emphasized as optimal in the sense that it represents the most suitable choice among all admissible exponents.  Utilizing the De Giorgi iteration along with localization arguments and taking  advantage of compactness of $\overline{\Omega},$ they  established the boundedness of weak solutions for the following problem
\begin{eqnarray}\label{eq1'}
	-\operatorname{div}\mathcal{A}(x,u,\nabla u) =\mathcal{B}(x,u,\nabla u) \quad \text{in } \Omega,
\end{eqnarray} 
in both subcritical and critical growth cases, under the following assumptions on $\mathcal{A}$ and $\mathcal{B}$:  
\begin{enumerate}[label=\textnormal{(H)},ref=\textnormal{H}]
	\item\label{H}
	$\mathcal{A} : \Omega\times\R\times\R^N\to \R^N$ and $\mathcal{B} : \Omega \times \R\times \R^N\to \R$ are Carath\'eodory functions such that
	\begin{enumerate}[label=\textnormal{(\roman*)},ref=\textnormal{D$_1$(\roman*)}]
		\item
		$ |\mathcal{A}(x,t,\xi)|
		\leq \alpha_1 \left[|t|^{\frac{p^*(x)}{p'(x)}}+a(x)^{\frac{N-1}{N-q(x)}}|t|^{\frac{q^*(x)}{q'(x)}}+|\xi|^{p(x)-1} +a(x)|\xi|^{q(x)-1}+1\right],$ 
		\item
		$ \mathcal{A}(x,t,\xi)\cdot \xi \geq \alpha_2 \left[|\xi|^{p(x)} +a(x)|\xi|^{q(x)}\right]-\alpha_3\left[|t|^{r(x)}+a(x)^{\frac{s(x)}{q(x)}}|t|^{s(x)}+1\right],$
		\item
		$ |\mathcal{B}(x,t,\xi)|
		\leq \beta  \left[|t|^{r(x)-1}+a(x)^{\frac{s(x)}{q(x)}}|t|^{s(x)-1}+|\xi|^{\frac{p(x)}{r'(x)}} + a(x)^{\frac{1}{q(x)}+\frac{1}{s'(x)}}|\xi|^{\frac{q(x)}{s'(x)}} +1 \right],$
	\end{enumerate}
for a.\,a.\,$x\in \Omega$ and for all $(t,\xi) \in \R\times\R^N$, where $\alpha_1,\alpha_2,\alpha_3$, $\beta$ are positive constants, $r,s\in C(\overline{\Omega})$ satisfy $p(x) <r(x) \leq p^*(x)$ and $q(x) <s(x) \leq q^*(x)$ for all $x \in \overline{\Omega},$ and $t'(\cdot):=\frac{t(\cdot)}{t(\cdot)-1}$ for $t\in \{q,s\}$.
\end{enumerate}
In particular, they obtained an \textit{a priori} bound for solutions in the subcritical case, namely, $r(x) < p^*(x)$ and $s(x) < q^*(x)$ for all $x \in \overline{\Omega}$.
It is worth mentioning that the idea for treating critical growth case in \cite{HW2022} originates from \cite{HKWZ}, which initially proved the global boundedness and H\"older continuity to critical $p(\cdot)$-Laplace equations. Following up on the work in \cite{HW2022}, Ha-Ho \cite{HH25} established the embedding \eqref{CE} for any open domain in $\mathbb{R}^N$, and applied it to obtain the existence and concentration of solutions for problem \eqref{eq1} with $\mathcal{A}$ of the form \eqref{p(.)-La} and $\mathcal{B}$ involving $\mathcal{G}^*$. They further emphasized that the potential $V$ may vanish, be coercive or be a positive constant (see condition \eqref{O} in Section~\ref{Pre}), after discussing several well-known studies on this topic \cite{Bar.Wang2001,BW.95,Sun-Wu-2014}. 

                                                                                                                \vskip5pt
                                                                                                                In general, a global $L^\infty$-estimate provides useful bounds for solutions to problem \eqref{eq1}, playing a prominent role and serving as a foundation for studying qualitative properties of solutions, such as regularity and Harnack-type inequalities. The global boundedness of solutions to  elliptic problems with constant exponents on unbounded domains, including  $\mathbb{R}^N,$ has been established in several works, see \cite{CT24,HL08,Li90}. A common technique for obtaining such results is the Moser iteration method, which relies on the  H\"older inequality to derive improved estimates for Lebesgue norms. However, this technique is generally not applicable to problems driven by nonhomogeneous operators—such as double phase operators (including those involving the constant exponents) or the $p(\cdot)$-Laplacian—due to the lack of H\"older-type inequality in integral form (compare Proposition \ref{prop.Holder}).
   
\vskip5pt
Based on the above motivation, we aim to establish two global boundedness results for weak solutions to problem \eqref{eq1} under assumptions \eqref{H1} (see Section~\ref{Sub}) and \eqref{H2} (see Section~\ref{Cri}) on $\mathcal{A}$ and $\mathcal{B}$, which represent modifications of \eqref{H} corresponding to the subcritical and critical cases, respectively, as well as assumption \eqref{O} on the potential $V$. 
One of the main tools for deriving these results is the Sobolev-type embedding on localized domains (see Proposition \ref{RmkEm}).  First, we obtain an \textit{a priori} bound for weak solutions in the subcritical growth case  by applying the De Giorgi iteration along with a refined localization technique adapted to the whole space $\mathbb{R}^N$. By using the H\"older inequality and then Sobolev-type embedding incorporating the relations between modular and its corresponding norm, we can estimate the truncated energies by the measure of level sets and derive the \textit{a priori} bound. As a consequence, we also derive the decay property of weak solutions, which supports the study of the existence of decaying positive solutions in $\mathbb{R}^N.$ Second, we establish  the global boundedness of weak solutions in the case of critical growth. To achieve this, we again apply the De Giorgi iteration, incorporating a recursive scheme involving the truncated energy of $\mathcal{G}^*$, and use a localization technique related to $\frac{p^*(\cdot)}{q(\cdot)}$. Note that in this case, we cannot use H\"older's inequality to estimate the truncated energy. However, by using appropriate estimates of modular by norm, we can still establish the global boundedness in this case.  

\vskip5pt
Having established the boundedness of solutions, it is natural to investigate their H\"older continuity. We observe that the H\"older continuity of solutions to the double phase problem \eqref{eq1'},
where $\Omega$ is a bounded Lipschitz domain, was established only recently in \cite{RK25}. By combining this result with our boundedness result,  we can deduce the local H\"older continuity of solutions to problem~\eqref{eq1} in the case $V\in L^\infty(\R^N).$ 

\vskip5pt 
In addition to applications in the study qualitative properties of solutions, our result can be useful in another direction. In this paper, we apply the aforementioned \textit{a priori} bounds to investigate the existence results for supercritical double phase problems. One of the main challenges in dealing with supercritical problems is that the corresponding energy functional is generally not well-defined on the Sobolev space associated with the main operator when the underlying domain is a general Lipschitz domain. To overcome this difficulty, a widely used method—originally introduced by Rabinowitz \cite{Rab74} for a class of Laplace equations—is to apply a truncation to the reaction term so as to reduce the problem to the subcritical case. Then,  an $L^\infty$-estimate is established for solutions to the truncated problem, ensuring that they also solve the original one. Since then, this approach has been extensively adopted to establish existence results for supercritical problems driven by the $p$-Laplacian; see, e.g., \cite{Cor-Fig06, Zhao24} for bounded domains and \cite{Fig-Fur07, LW13} for problems posed on $\mathbb{R}^N$.
 It is worth noting that in most existing works, the Moser iteration  is employed to obtain the desired  $L^\infty$-estimate. However, as mentioned earlier, this approach generally does not work for certain non-homogeneous operators, such as the $p(\cdot)$-Laplacian or  double phase operators. In \cite{Zhao24}, Zhao established such an $L^\infty$-estimate for solutions to truncated problems of a class of supercritical $p(\cdot)$-Laplace equations using the De Giorgi iteration. However, to the best of our knowledge, there are no known existence results for supercritical double phase problems, especially for variable exponents on $\R^N$.

\vskip5pt 
The paper is organized as follows. In Section~\ref{Pre}, we present the definition and basic properties of the solution space for problem~\eqref{eq1}, as established in \cite{HH25}, and prove a crucial embedding result (Proposition \ref{RmkEm}) that will be used in subsequent sections. In Section~\ref{Sub}, we state and prove \textit{a priori} bounds and the decay at infinity of solutions to problem~\eqref{eq1} under subcritical growth assumptions. Section~\ref{Cri} is devoted to the global boundedness of solutions in the critical growth case, where we also provide the corresponding proof. As an application of the main results, the final section addresses the existence of solutions to supercritical double phase problems, both in $\mathbb{R}^N$ and in bounded domains.

\section{Preliminaries}\label{Pre}

In this section, we review the main properties of Musielak-Orlicz and Musielak-Orlicz Sobolev spaces which serve as the solution spaces for problem~\eqref{eq1}. Most of these properties were established in \cite{crespo2022new,HW2022,HH25}. For a comprehensive survey on the topic, we refer the reader to \cite{Diening,HH.book,Mus}.

\vskip5pt

Let $\Omega$ be an open domain in $\RN$. Throughout the article, for a function $h\in C(\overline{\Omega})$ we denote
$$h_\Omega^-:= \inf_{x\in \Omega}h(x)\ \ \text{and}\ \  h_\Omega^+:= \sup_{x\in \Omega}h(x).$$
If $\Omega$ is implicitly understood, we shall write $h^-$ and $h^+$ in place of $h_\Omega^-$ and $h_\Omega^+$, respectively. Denote
\begin{align*}
	C_+(\overline{\Omega}):=\{v\in C(\overline{\Omega}): \,  1<v^-\leq v(\cdot) \leq v^+<\infty \}
\end{align*}
and
\begin{equation*}
	\mathcal{M}(\Omega):=\{m: \Omega \to \mathbb{R}: \, m ~ \text{is  Lebesgue measurable}~ \text{in}~ \Omega \}.
\end{equation*}
For any function $h\in \mathcal{M}(\Omega)$, we denote
\begin{equation*}
	h_+(\cdot):=\max\{h(\cdot),0\},~ 	h_-(\cdot):=-\min\{h(\cdot),0\}.
\end{equation*}
For $t\in C(\overline{\Omega})$ with $1<t(\cdot)<N$,  denote $$t'(\cdot):=\frac{t(\cdot)}{t(\cdot)-1}\  \text{and}\ \ t^\ast(\cdot) := \frac{Nt(\cdot)}{N-t(\cdot)}.$$ 
 For $E\subset\R^N$, $|E|$ denotes the Lebesgue measure of $E$. 
\subsection{A class of Musielak-Orlicz Sobolev spaces}${}$ \label{Subsec.M.O.S space}

\vskip5pt

Define $\Psi:\, \overline{\Omega}\times [0,\infty)\to [0,\infty)$ as 
\begin{align}\label{Psi}
	\Psi(x,t):=w_1(x)t^{\alpha(x)}+w_2(x)t^{\beta(x)}
	\quad\text{for } (x,t)\in \overline{\Omega}\times [0,\infty),
\end{align}
where $\alpha,\,\beta\in C_+(\close)$ with $\alpha(\cdot)<\beta(\cdot)$, $0 <  w_1(\cdot) \in \Lp{1}\cup L^\infty(\Omega)$ and $0 \leq w_2(\cdot) \in \Lp{1}\cup L^\infty(\Omega)$.

\noindent Define the modular $\rho_{\Psi}$ associated with $\Psi$ as
\begin{align*}
	\rho_{\Psi}(u): =  \into \Psi (x,|u(x)|)\,\diff x .
\end{align*}
The corresponding Musielak-Orlicz space $\Lp{\Psi}$ is defined by 
\begin{align*}
	L^{\Psi}(\Omega):=\left \{u\in\mathcal{M}(\Omega):\,\rho_{\Psi}(u) < \infty \right\},
\end{align*}
endowed with the norm
\begin{align*}
	\|u\|_{L^{\Psi}(\Omega)}:= \inf \left \{ \tau >0 :\, \rho_{\Psi}\left(\frac{u}{\tau}\right) \leq 1  \right \}.
\end{align*}
Then, $\Lp{\Psi}$ is a separable, uniformly convex and reflexive Banach space (see \cite[Theorems 3.3.7, 3.5.2
and 3.6.6]{HH.book}). Moreover, we have the following relation between the modular $\rho_{\Psi}$ and the norm $\|\cdot\|_{L^{\Psi}(\Omega)}$ (see \cite[Proposition 2.3]{HH25}).
\begin{proposition}\label{prop.nor-mod.D}
	Let $u, u_n\in\Lp{\Psi}$ ($n\in\N$). Then, the following assertions hold:
	\begin{enumerate}
		\item[\textnormal{(i)}]
		if $u\neq 0$, then $\|u\|_{L^{\Psi}(\Omega)}=\lambda$ if and only if $ \rho_{\Psi}(\frac{u}{\lambda})=1$;
		\item[\textnormal{(ii)}]
		$\|u\|_{L^{\Psi}(\Omega)}<1$ (resp.\,$>1$, $=1$) if and only if $ \rho_{\Psi}(u)<1$ (resp.\,$>1$, $=1$);
		\item[\textnormal{(iii)}]
		if $\|u\|_{L^{\Psi}(\Omega)}<1$, then $\|u\|_{L^{\Psi}(\Omega)}^{ \beta^+}\leqslant \rho_{\Psi}(u)\leqslant\|u\|_{L^{\Psi}(\Omega)}^{ \alpha^-}$;
		\item[\textnormal{(iv)}]
		if $\|u\|_{L^{\Psi}(\Omega)}>1$, then $\|u\|_{L^{\Psi}(\Omega)}^{ \alpha^-}\leqslant \rho_{\Psi}(u)\leqslant\|u\|_{L^{\Psi}(\Omega)}^{ \beta^+}$;
		\item[\textnormal{(v)}] $\|u_n\|_{L^{\Psi}(\Omega)}\to 0$ as $n\to \infty$ if and only if $\rho_{\Psi}(u_n)\to 0$ as $n\to \infty$.
	\end{enumerate}
\end{proposition}

The Musielak-Orlicz Sobolev space  $W^{1,\Psi}(\Omega)$ is defined by
\begin{align*}
	W^{1,\Psi}(\Omega):=\left\{u \in L^{\Psi}(\Omega):|\nabla u| \in L^{\Psi}(\Omega)\right\},
\end{align*}
endowed with the norm
\begin{align*}
	\|u\|_{W^{1,\Psi}(\Omega)}:=\inf \left\{\tau>0: \rho_{1,\Psi} \left(\frac{u}{\tau}\right) \leq 1\right\},
\end{align*}
where $\rho_{1,\Psi}:W^{1,\Psi}(\Omega) \to \mathbb{R}$ is the modular defined by
\begin{align*}
	\rho_{1,\Psi}(u):=  \int_\Omega \big[\Psi(x,|\nabla u|)+\Psi(x,|u|)\big]\diff x.
\end{align*}
We also have the relation between the modular $\rho_{1,\Psi}$ and the norm $\|\cdot\|_{W^{1,\Psi}(\Omega)}$ as follows.
\begin{proposition}[\cite{HH25}]\label{Prop_m-n_W(1,H)}  
	For any $u \in W^{1,\Psi}(\Omega) $, we have
	\begin{enumerate}
		\item[\textnormal{(i)}]
		if $u\neq 0$, then $\|u\|_{W^{1,\Psi}(\Omega)}=\lambda$ if and only if $ \rho_{1,\Psi}(\frac{u}{\lambda})=1$;
		\item[\textnormal{(ii)}]
		$\|u\|_{ W^{1, \Psi}(\Omega)}<1$ (resp.\, $>1$, $=1$)  if and only if	$\rho_{1,\Psi}(u)<1$ (resp.\,$>1, $=1$)$;
		\item[\textnormal{(iii)}]
		if $\|u\|_{ W^{1, \Psi}(\Omega)}<1 $, then $\|u\|^{\beta^+}_{ W^{1, \Psi}(\Omega)} \leq \rho_{1,\Psi}(u) \leq \|u\|^{\alpha^-}_{ W^{1, \Psi}(\Omega)}$;
		\item[\textnormal{(iv)}]
		if $\|u\|_{ W^{1, \Psi}(\Omega)}>1 $, then $\|u\|^{\alpha^-}_{ W^{1, \Psi}(\Omega)} \leq \rho_{1,\Psi}(u) \leq \|u\|^{\beta^+}_{ W^{1, \Psi}(\Omega)}$.
	\end{enumerate}
\end{proposition}
We denote by $W_0^{1,\Psi}(\Omega)$ the closure of $C_c^\infty(\Omega)$ in $W^{1,\Psi}(\Omega)$. As shown in \cite[Theorem 6.1.4]{HH.book}, $W^{1,\Psi}(\Omega)$ and $W_0^{1,\Psi}(\Omega)$ are separable, uniformly convex and reflexive Banach spaces. 

Note that when $\Psi(x,t)=t^{\alpha(x)},$ the spaces $\Lp{\Psi}$, $W^{1,\Psi}(\Omega)$ and $W_0^{1,\Psi}(\Omega)$ become the well-studied generalized Lebesgue-Sobolev space $\Lp{\alpha(\cdot)}$, $W^{1,\alpha(\cdot)}(\Omega)$ and $W_0^{1,\alpha(\cdot)}(\Omega)$, respectively (see e.g., \cite{Diening}), and in this case we shall write $\|\cdot\|_{L^{\alpha(\cdot)}(\Omega)}$ in place of $\|\cdot\|_{L^{\Psi}(\Omega)}$.  The next two results on generalized Lebesgue spaces will be frequently used in the next sections.
\begin{proposition}[\cite{fanzhao2001}]\label{prop.Holder} It holds that
	\begin{equation*}
		\left|\int_\Omega uv \diff x\right|\leq 2 \|u\|_{L^{\alpha(\cdot)}(\Omega)}\|v\|_{L^{\alpha'(\cdot)}(\Omega)},\quad \forall u\in L^{\alpha(\cdot)}(\Omega),\ \forall v\in  L^{\alpha'(\cdot)}(\Omega).
	\end{equation*}
\end{proposition}
\begin{proposition}[\cite{Diening}]\label{DHHR} Let $\alpha,\beta\in C_+(\overline{\Omega})$ such that $\alpha(x)\le \beta(x)$ for all $x\in\overline{\Omega}.$ Then, we have 
	$$\|u\|_{L^{\alpha(\cdot)}(\Omega)}\le
	2\left(1+|\Omega|\right)\|u\|_{L^{\beta(\cdot)}(\Omega)},\quad\forall u\in L^{\alpha(\cdot)}(\Omega)\cap L^{\beta(\cdot)}(\Omega).$$
\end{proposition}


\subsection{The Musielak-Orlicz Sobolev spaces $W^{1,\mathcal{H}}_V(\RN)$ and $X_V$}${}$ \label{subsec.sol.space}

\vskip5pt
In this subsection, we present the definition and basic properties of the solution space for problem~\eqref{eq1} established in \cite{HH25}. Let us denote
\begin{equation}\label{def_H}
	\mathcal{H}(x,t):=t^{p(x)} +a(x)t^{q(x)} \quad \text{for} \ (x,t) \in \RN \times [0,\infty),
\end{equation}
where the functions $p,q,a$ satisfy the following assumption:
\begin{enumerate}[label=\textnormal{(D)},ref=\textnormal{D}]
	\item\label{D}
	$p,q\in C^{0,1}(\RN) $ such that $1<p(x)<q(x)<N$  for all $x\in \RN,$ $ \displaystyle \left(\frac{q}{p}\right)^+<1+\frac{1}{N}$ and $0  \leq a(\cdot) \in C^{0,1}(\RN) \cap L^\infty(\RN) $.
\end{enumerate}
Let $V\in L^1_{\loc}\left(\R^N\right)$ be such that $V(\cdot) \geq 0$ and $V\ne 0$. We define the space $W^{1,\mathcal{H}}_V(\RN)$ as
\begin{align*}
	W^{1,\mathcal{H}}_V(\RN):=\left\{u\in W_{\loc}^{1,1}(\RN):~ \rho_V(u)<\infty\right\},
\end{align*}
where the modular $\rho_V$ is defined by
\begin{align*}
	\rho_V(u):=\int_{\RN} \big[\mathcal{H}\left(x,|\nabla u|\right) + V(x)\mathcal{H}\left(x,|u|\right)\big] \diff x\quad\text{for}\ u\in W_{\loc}^{1,1}(\RN).
\end{align*}
Then, $W^{1,\mathcal{H}}_V(\RN)$ is a normed space with the norm
\begin{align}\label{norm}
	\|u\|:=\inf\left\{ \tau >0: \rho_V\left(\frac{u}{\tau}\right) \leq 1\right\}.
\end{align}
On this space we have
\begin{equation}\label{modular-norm.XV}
	\rho_V\left(\frac{u}{\|u\|}\right)=1, \quad \forall u \in W_V^{1,\mathcal{H}}(\RN) \setminus \{0\},
\end{equation}
and for $u\in W^{1,\mathcal{H}}_V(\RN)$ with $\operatorname{supp}(u)\subset\Omega$, it holds that
\begin{equation}\label{m-n}
	\min \left\{\|u\|^{p_\Omega^-},\|u\|^{q_\Omega^+}\right\}\leq \rho_V(u)\leq \max \left\{\|u\|^{p_\Omega^-},\|u\|^{q_\Omega^+}\right\}.
\end{equation}

Next we present essential embeddings established in \cite{HH25,HW2022}. Define $\mathcal{F}:\RN \times [0,\infty)\to [0,\infty)$ as
\begin{equation}\label{def_F}
	\mathcal{F}(x,t):=t^{r(x)} +a(x)^{\frac{s(x)}{q(x)}}t^{s(x)} ~\text{for} ~ (x,t) \in \RN \times [0,\infty),
\end{equation}
where $r,s \in C_+(\RN)$ satisfying $r(\cdot)< s(\cdot)$ in $\RN$. 
For the case of $\essinf_{x \in \R^N}V(x)>0$, we have the following.
\begin{proposition}[\cite{HH25}]\label{P.E1}
	Let \eqref{D} hold, and let $V\in L^1_{\loc}\left(\R^N\right)$ be such that $\essinf_{x \in \R^N}V(x)>0$.  Then, $W^{1,\mathcal{H}}_V(\RN)$ is a separable reflexive Banach space satisfying
	\begin{equation}\label{Embed.W_V0}
		W^{1,\mathcal{H}}_V(\RN)\hookrightarrow W^{1,\mathcal{H}}(\R^N) \hookrightarrow L^{\mathcal{F}}(\R^N)
	\end{equation}
for any $r,s \in C(\RN)$ satisfying $p(x) \leq r(x) \leq p^\ast(x)$ and  $q(x) \leq s(x) \leq q^\ast(x)$ for all  $x \in \RN.$  
\end{proposition}
When $V\equiv 1$ and the underlying domain is bounded, we have the following.
\begin{proposition}[\cite{HW2022}]\label{P.Eb}
	Let \eqref{D} hold and $\mathcal{F}$ be given by \eqref{def_F}. Let $\Omega$ be a bounded Lipschitz domain in $\RN$. Then, it holds that
	\begin{equation}\label{EC}
		W^{1,\mathcal{H}}(\Omega) \hookrightarrow  L^{\mathcal{F}}(\Omega)
	\end{equation}
	if  $1\leq r(x)\leq p^\ast(x)$ and $1\leq s(x)\leq q^\ast(x)$ for all $x\in\overline{\Omega}$. Furthermore, the embedding \eqref{EC} is compact if $1\leq r(x)< p^\ast(x)$ and $1\leq s(x)< q^\ast(x)$ for all $x\in\overline{\Omega}$.
\end{proposition}
Furthermore, we present an inequality, which  is essential for localization arguments in the next sections. Let $Q_{x_0,\varepsilon}$ denote the cube centered at $x_0\in\R^N$ with edges parallel coordinate axes and of length $\varepsilon$.
\begin{proposition}\label{RmkEm} Let \eqref{D} hold, and let $\varepsilon>0$ be given. Then, for any $\mathcal{F}$ given by \eqref{def_F} with $p(x)\leq r(x)\leq p^*(x)$ and $q(x)\leq s(x)\leq q^*(x)$ for all $x \in \RN$, it holds that
	\begin{equation*}
		\|u\|_{L^{\mathcal{F}}(Q_{x_0,\varepsilon})}\leq \bar{C}\|u\|_{W^{1,\mathcal{H}}(Q_{x_0,\varepsilon})},\quad \forall u\in W^{1,\mathcal{H}}(\RN),\ \forall x_0\in\RN,
	\end{equation*}
	where $\bar{C}>0$ depends only on $\varepsilon$, $N$, $p^+$, $p^-$, $q^+$, $q^-$, $\|\nabla p\|_\infty$, $\|\nabla q\|_\infty$, $\|a\|_\infty$ and $\|\nabla a\|_\infty$.
\end{proposition}
\begin{proof}
	Let $\mathcal{G}^*$ be the critical form of $\mathcal{H}$ defined by 
	\begin{align}\label{def.G}
		\mathcal{G}^*(x,t):=t^{p^*(x)}+a(x)^{\frac{q^*(x)}{q(x)}}t^{q^*(x)} ~ \text{for} ~ x \in \RN ~ \text{and } ~ t \in [0,\infty).
	\end{align}
	A careful inspection of the proof of Theorem 2.7 in \cite{HH25} shows that 
	\begin{equation*} 
		\int_{Q_{0,\varepsilon}} \mathcal{G}^\ast(x,|u|)\diff x  \leq \widetilde{C} \left[\int_{Q_{0,\varepsilon}}\mathcal{H}(x,|\nabla u|)\diff x +\int_{Q_{0,\varepsilon}}\mathcal{H}(x, |u|)\diff x+1\right]^{{(q^\ast)}^+}, \quad \forall u \in W^{1,\mathcal{H}}(\RN),
	\end{equation*}
	where 	$\widetilde{C}=\widetilde{C}(\varepsilon,N,p^+,p^-,q^+,q^-,\|\nabla p\|_\infty,\|\nabla q\|_\infty,\|a\|_\infty,\|\nabla a\|_\infty)$. Then, by a translation we have
	\begin{equation*} 
		\int_{Q_{x_0,\varepsilon}} \mathcal{G}^\ast(x,|u|)\diff x  \leq \widetilde{C} \left[\int_{Q_{x_0,\varepsilon}}\mathcal{H}(x,|\nabla u|)\diff x +\int_{Q_{x_0,\varepsilon}}\mathcal{H}(x, |u|)\diff x+1\right]^{{(q^\ast)}^+}, \quad \forall u \in W^{1,\mathcal{H}}(\RN).
	\end{equation*}
	Since $\mathcal{F}(x,|u|)\leq \mathcal{G}^\ast(x,|u|)+\mathcal{H}(x,|u|)$ for a.a. $x\in\RN$, we obtain from the last estimate that
	\begin{equation*}
		\int_{Q_{x_0,\varepsilon}} \mathcal{F}(x,|u|)\diff x  \leq (\widetilde{C}+1) \left[\int_{Q_{x_0,\varepsilon}}\mathcal{H}(x,|\nabla u|)\diff x +\int_{Q_{x_0,\varepsilon}}\mathcal{H}(x, |u|)\diff x+1\right]^{{(q^\ast)}^+}, \quad \forall u \in W^{1,\mathcal{H}}(\RN).
	\end{equation*}
	Hence, the conclusion follows.
\end{proof}

 In order to broaden the class of potentials $V$, in \cite{HH25} the authors introduced the following assumption on the main operator:
 \begin{enumerate}[label=\textnormal{(O)},ref=\textnormal{O}]
 	\item\label{O} The functions $p,q,a$ fulfill \eqref{D}; $ 0 \leq V(\cdot) \in L^1_{\loc}(\RN)$, and one of two following conditions holds:
 	\begin{enumerate}
 		\item[(i)] there exists $K_0>0$ such that $E_V:=\{x \in \RN:V(x) \leq K_0\} \not = \emptyset$ and $|E_V| < \infty$;
 		\item[(ii)] there exists $R>0$ such that $\displaystyle \inf_{x\in B_R^c} V(x)=V_0>0$. 
 	\end{enumerate}
Furthermore, when $\essinf_{x \in \R^N}V(x)=0$, there exists $p_\infty \in (1,N)$ such that 
\begin{align*}
	\sup_{x \in \RN} |p(x)-p_\infty|\log(e+|x|)<\infty.
\end{align*}
 \end{enumerate}
We denote by $X_V$ the closure of $C_c^\infty(\R^N)$ in $W_V^{1,\mathcal{H}}(\R^N)$. The next result was proved in \cite{HH25}.
\begin{proposition}\label{P.E2}
	Let \eqref{O} hold. Then, the following assertions hold:
\begin{enumerate}
	\item[\textnormal{(i)}] $X_V$ is a separable reflexive Banach space.
	\item[\textnormal{(ii)}] For $r,s \in C(\RN)$ satisfying $p(x) \leq r(x) \leq p^\ast(x)$ and $q(x) \leq s(x) \leq q^\ast(x)$ for all  $x \in \RN$, we have
	\begin{equation}\label{E-X}
		X_V \hookrightarrow W^{1,\mathcal{H}}(\R^N) \hookrightarrow L^{\mathcal{F}}(\R^N).
	\end{equation} 
\end{enumerate}
\end{proposition}

The following property on our Musielak-Orlicz Sobolev spaces is necessary for showing the boundedness of weak solutions to problem~\eqref{eq1}.
\begin{lemma}\label{lem.(u-k)+}
	Let \eqref{O} hold, and let $X$ be $X_V$ or $W^{1,\mathcal{H}}_V(\RN)$. Then, $(u-k)_+,(u+k)_- \in X$ for any $u \in X$ and $k \geq 0$. 
\end{lemma}
\begin{proof}
	Let $u\in X$ and $k\geq 0$. We only need to show that $v:=(u-k)_+\in X$ since 	$(u+k)_-=(-u+k)_+$. Note that
	\begin{equation} \label{nab.v}
		\nabla v =\begin{cases}
			\nabla u & ~\text{if} ~ u > k, \\
			0  & ~\text{if} ~  u \leq  k.
		\end{cases}
	\end{equation} The conclusion for the case of $X=W^{1,\mathcal{H}}_V(\RN)$ is trivial thanks to the following estimates 
	$$|v|\leq |u|\ \ \text{and}\ \ |\nabla v|\leq |\nabla u|.$$ Let us consider the other case, namely, $X= X_V$.  
	To get the conclusion, we need to show the existence of a sequence $\{v_n\}_{n \in \N} \subset  C_c^\infty(\RN) $ such that
	\begin{equation}\label{vn}
		\int_{\R^N} \mathcal{H}(x,|\nabla v_n - \nabla v|) \diff x + 	\int_{\R^N} V(x) \mathcal{H}(x,|v_n -  v|) \diff x \to 0 ~ \text{as}~ n \to \infty.
	\end{equation}
	By the definition of $X_V$ and Proposition~\ref{prop.nor-mod.D}, we find a sequence $\{\varphi_n\}_{n \in \N} \subset C_c^\infty(\RN)$ such that 
	\begin{equation}\label{phin}
		\int_{\R^N} \mathcal{H}(x,|\nabla \varphi_n - \nabla u|) \diff x + 	\int_{\R^N} V(x) \mathcal{H}(x,|\varphi_n -  u|) \diff x \to 0 \quad \text{as} \ \ n \to \infty.
	\end{equation}
	Let $\rho_n=(\varphi_n-k)_+$ for each $n \in \mathbb{N}$. It is not difficult to see that
	\begin{align*}
		|\rho_n-v| \leq |\varphi_n - u| \quad \text{a.e. in } \RN.
	\end{align*}
	This and \eqref{phin} yield  
	\begin{equation}\label{L2.7-2.15}
	\int_{\R^N} V(x) \mathcal{H}(x,|\rho_n-v|) \diff x \to 0\quad \text{as}  \ \ n \to \infty.
	\end{equation} Moreover, let $\chi(x)$ be the characteristic function on the interval $(k,\infty)$. The  following estimates hold in view of \eqref{nab.v}:
	\begin{align*}
		\left| \nabla \rho_n - \nabla v \right|^{p(x)} &  =  \left| \chi(\varphi_n)\nabla \varphi_n - \chi(u)\nabla u \right|^{p(x)}\\
		& \leq 2^{p^+-1} \left[ |\chi (\varphi_n)|^{p(x)} \left|\nabla \varphi_n-\nabla u\right|^{p(x)} + \Big|\chi(\varphi_n)-\chi(u)\Big|^{p(x)}|\nabla u|^{p(x)}\right]\\
		& \leq 2^{p^+-1} \left[\left|\nabla \varphi_n-\nabla u\right|^{p(x)} + \Big|\chi(\varphi_n)-\chi(u)\Big||\nabla u|^{p(x)}\right]
	\end{align*}
	and 
	\begin{align*}
		a(x)\left| \nabla \rho_n - \nabla v \right|^{q(x)} &  = a(x) \left| \chi(\varphi_n)\nabla \varphi_n - \chi(u)\nabla u \right|^{q(x)}\\
		& \leq 2^{q^+-1} \left[ |\chi (\varphi_n)|^{q(x)}a(x) \left|\nabla \varphi_n-\nabla u\right|^{q(x)} + \Big|\chi(\varphi_n)-\chi(u)\Big|^{q(x)}a(x)|\nabla u|^{q(x)}\right]\\
		& \leq 2^{q^+-1} \left[a(x) \left|\nabla \varphi_n-\nabla u\right|^{q(x)} + \Big|\chi(\varphi_n)-\chi(u)\Big|a(x)|\nabla u|^{q(x)}\right]
	\end{align*}
	for a.a. $x \in \RN$.
	From the last two estimates we obtain
	\begin{align*}
		\mathcal{H}\left(x,\left| \nabla \rho_n - \nabla v \right|\right) \leq 2^{q^+-1} \big[\mathcal{H}\left(x,\left|\nabla \varphi_n-\nabla u\right|\right) + \left|\chi(\varphi_n)-\chi(u)\right| \mathcal{H}(x,|\nabla u|)\big]
	\end{align*}
	for a.a. $x \in \RN$. Using this, \eqref{phin} and the Lebesgue dominated convergence theorem we deduce
	\begin{equation} \label{grad.conver}
		\int_{ \RN}	\mathcal{H}\left(x,|\nabla \rho_n -\nabla v|\right) \diff x \to 0 \quad \text{as} \ \ n \to \infty.
	\end{equation}
	Combining \eqref{L2.7-2.15} with \eqref{grad.conver} gives
	\begin{equation}\label{rho_n-v}
		\int_{\R^N} \mathcal{H}(x,|\nabla \rho_n - \nabla v|) \diff x + 	\int_{\R^N} V(x) \mathcal{H} \left(x,|\rho_n -  v|\right) \diff x \to 0 \quad  \text{as} \ \ n \to \infty.
	\end{equation}
	
	Next step is to approximate $\{\rho_n\}_{n \in \N}$ with a sequence of smooth functions by the method of mollifiers. Let $\eta$ be a standard mollifier in $\RN$ and $n$ be fixed. For each $i\in \mathbb{N}$, define $\eta_i(x)=i^N \eta(ix)$ and the convolution
	\begin{equation*}
		\psi_i^{(n)}(x) = \int_{\RN} \eta_i(x-y)\rho_n(y)\diff y.
	\end{equation*}
	By the properties of mollifiers (see \cite{Brezis2011}), we have $\psi_i^{(n)} \in C_c^\infty(\RN)$, and furthermore, $	\psi_i^{(n)}(x) \to \rho_n(x) $ and $\nabla \psi_i^{(n)}(x) \to \nabla \rho_n(x)$ as $i \to \infty$ on  compact subsets of $\RN$. This and the assumption $V\in L^1_{\loc}\left(\R^N\right)$ yield
	\begin{equation*}
		\int_{\R^N} \mathcal{H}(x,|\nabla 	\psi_i^{(n)} - \nabla \rho_n|) \diff x + 	\int_{\R^N} V(x) \mathcal{H}(x,|\psi_i^{(n)} -   \rho_n|) \diff x \to 0 ~ \text{as}~ i \to \infty.
	\end{equation*}
	The above limit implies that we can find $i_n \in \mathbb{N}$ large such that
	\begin{equation*}
		\int_{\R^N} \mathcal{H}(x,|\nabla 	\psi_{i_n}^{(n)} - \nabla \rho_n|) \diff x + 	\int_{\R^N} V(x) \mathcal{H}(x,|\psi_{i_n}^{(n)} -   \rho_n|) \diff x <\frac{1}{n}.
	\end{equation*}
	Combining this and \eqref{rho_n-v} we get
	\begin{equation*}
		\int_{\R^N} \mathcal{H}(x,|\nabla 	\psi_{i_n}^{(n)} - \nabla v|) \diff x + 	\int_{\R^N} V(x) \mathcal{H}(x,|\psi_{i_n}^{(n)} -   v|) \diff x \to 0 ~ \text{as}~ n \to \infty.
	\end{equation*}
	Thus, \eqref{vn} has been shown by taking $v_n=	\psi_{i_n}^{(n)}$. The proof is complete.
\end{proof}

Finally, we recall a recursion lemma, which is essential for showing the boundedness of solutions to problem~\eqref{eq1} via the De Giorgi iteration in the next sections.

\begin{lemma}[\cite{Ho-Sim-2015}]\label{leRecur}
	Let $\{Z_n\}, n=0,1,2,\ldots,$ be a sequence of nonnegative numbers, satisfying the recursion inequality
	\begin{align*}
		Z_{n+1} \leq K b^n \left (Z_n^{1+\mu_1}+ Z_n^{1+\mu_2} \right ) , \quad n=0,1,2, \ldots,
	\end{align*}
	for some $b>1$, $K>0$ and $\mu_2\geq \mu_1>0$. If $$Z_{0}\leq \min\left(1,(2K)^{\frac{-1}{\mu_1}}\ b^{\frac{-1}{\mu_1^{2}}}\right) $$ or
		\begin{align*}
		Z_0 \leq \min  \left((2K)^{-\frac{1}{\mu_1}} b^{-\frac{1}{\mu_1^2}}, (2K)^{-\frac{1}{\mu_2}}b^{-\frac{1}{\mu_1 \mu_2}-\frac{\mu_2-\mu_1}{\mu_2^2}}\right),
	\end{align*}
	then $Z_n \to 0$ as $n \to \infty$.
\end{lemma}

In the next sections, we shall also frequently use the following notations:
\begin{itemize}
	\item $f \ll g$ means $\inf_{x \in \RN}(g(x)-f(x))>0$.
	\item  $\N_0:=\N \cup \{0\}$.
	\item The notation $B_R$ denotes a ball centered at the origin with radius $R$ and $B_R^c:=\RN \setminus B_R$.
	\item $Q_\eta(y)$ denotes the cube centered at $y\in\RN$, with edges parallel to the coordinate axes and of length $\eta$.
		\item $h_{y,\eta}^+:=\sup_{x\in Q_\eta(y)}h(x)$ and $h_{y,\eta}^-:=\inf_{x\in Q_\eta(y)}h(x)$ for $h\in C(\R^N)$.
\end{itemize}


\section{A priori bounds and decay for subcritical problems }\label{Sub}

In this section, we shall establish an \textit{a priori} bound and the decay at infinity of weak solutions to problem~\eqref{eq1} when the growth of the nonlinear term is subcritical. Throughout this section, we always assume that the assumption \eqref{O} holds. 
We assume further that
\begin{enumerate}[label=\textnormal{(H$_1$)},ref=\textnormal{H$_1$}]
	\item\label{H1}
	$\mathcal{A}\colon\R^N\times\R\times\R^N\to \R^N$ and $\mathcal{B}\colon\R^N \times \R\times \R^N\to \R$ are Carath\'eodory functions such that
	\begin{enumerate}[label=\textnormal{(\roman*)},ref=\textnormal{H$_1$(\roman*)}]
		\item\label{H1i}
		$ |\mathcal{A}(x,t,\xi)|
		\leq M_1 \left[k(x)+|t|^{r(x)-1}+a(x)^{\frac{s(x)}{q(x)}}|t|^{s(x)-1}+|\xi|^{p(x)-1} +a(x)|\xi|^{q(x)-1}\right],$ 
		\item\label{H1ii}
		$ \mathcal{A}(x,t,\xi)\cdot \xi \geq M_2 \left[|\xi|^{p(x)} +a(x)|\xi|^{q(x)}\right]-M_3\left[|t|^{r(x)}+a(x)^{\frac{s(x)}{q(x)}}|t|^{s(x)}+1\right],$
		\item\label{H1iii}
		$ |\mathcal{B}(x,t,\xi)|
		\leq L \left[d(x)+|t|^{r(x)-1}+a(x)^{\frac{s(x)}{q(x)}}|t|^{s(x)-1}+|\xi|^{\frac{p(x)}{r'(x)}} + a(x)^{\frac{1}{q(x)}+\frac{1}{s'(x)}}|\xi|^{\frac{q(x)}{s'(x)}} \right],$
	\end{enumerate}
	for a.\,a.\,$x\in\R^N$ and for all $(t,\xi) \in \R\times\R^N$, where $M_1,M_2,M_3$, $L$ are positive constants; $r,s\in C^{0,1}(\RN)$ satisfy $1 \ll r \ll p^*$ and $1\ll s \ll q^*$; $k \in L_{\loc}^{\rho'(\cdot)}(\RN)$, $d \in L_{\loc}^{\sigma'(\cdot)}(\RN)$ with $\rho,\sigma \in C_+(\RN)\cap C^{0,1}(\RN)$ such that $\kappa:=\max\{\rho, \sigma\} \ll \frac{p^*}{q}$ and
	\begin{equation}\label{k-d.1}
		M:=\sup_{y \in \RN} \int_{Q_1(y)} \left(|k(x)|^{\rho'(x)}+|d(x)|^{\sigma'(x)}\right) \diff x<\infty.
	\end{equation}
\end{enumerate}
Denote
\begin{equation}\label{X}
	X:=\begin{cases}
		W^{1,\mathcal{H}}_V(\RN),&\ \ \text{when}\ \ \essinf_{x \in \R^N}V(x)>0\\
		X_V, &\ \ \text{when}\ \  \essinf_{x \in \R^N}V(x)=0.
	\end{cases}
\end{equation}
\begin{definition}\label{D.wsol1}
	By a weak solution of problem \eqref{eq1}, we mean a function $u\in X$ fulfilling
	\begin{align}\label{def.weak.sol}
		\int_{\R^N}\mathcal{A}(x,u,\nabla u)\cdot \nabla \varphi \diff x + \int_{\R^N} V(x)\Big(|u|^{p(x)-2}+a(x)|u|^{q(x)-2}\Big)u\varphi \diff x
		= \int_{\R^N} \mathcal{B}(x,u,\nabla u)\varphi \diff x 
	\end{align}
	for all $\varphi \in C_c^\infty(\R^N)$.
\end{definition}
By the definition \eqref{X} of $X$ and Propositions \ref{P.E1} and \ref{P.E2}, all integrals in \eqref{def.weak.sol} are well defined. Our first main result for subcritical problems is stated as follows.
\begin{theorem}\label{T.Sub1}
	Let \eqref{O} and \eqref{H1} hold.  Then, any weak solution $u\in X_V$ of problem \eqref{eq1} belongs to $L^\infty(\R^N)$ and fulfills
	\begin{align}\label{Sub-bound}
		\|u\|_{L^\infty(\RN)} \leq C \max \big\{\|u\|_{L^{\mathcal{F}}\left(\RN\right)}^{\tau_1},\|u\|_{L^{\mathcal{F}}\left(\RN\right)}^{\tau_2} \big\},
	\end{align}
	where $\mathcal{F}$ is given in \eqref{def_F}, and $C$, $\tau_1$ and $\tau_2$ are positive constants independent of $u$. Moreover, the decay of  $u$ at infinity holds, namely,
	\begin{equation} \label{decay}
		\lim\limits_{|x| \to \infty} u(x) = 0.
	\end{equation}
\end{theorem}
When $\essinf_{x \in \R^N}V(x)>0$, we have the same result for a possibly weaker solution as follows, noting  $X_V\subset W^{1,\mathcal{H}}_V(\RN)$.
\begin{theorem}\label{T.Sub2}
	Let \eqref{D} and \eqref{H1} hold, and let $ V\in L^1_{\loc}(\RN)$ satisfy $\essinf_{x \in \R^N}V(x)>0$. Then, the conclusions of Theorem~\ref{T.Sub2} remain valid for any weak solution $u\in W^{1,\mathcal{H}}_V(\RN)$ of problem \eqref{eq1}.
\end{theorem}
\begin{proof}[Proof of Theorems~\ref{T.Sub1}]

Let $u\in X_V$ be a weak solution of problem \eqref{eq1} and $\mathcal{F}$ be given in \eqref{def_F}. In order to obtain \eqref{Sub-bound} and  \eqref{decay}, we shall apply a localization method and the De Giorgi iteration, in which Lemma \ref{leRecur} is used. 
Note that  from the estimate $t^\alpha<t^\beta+1$ for all $t>0$ and all $\beta\geq \alpha>0$, by replacing $r$, $s$, $M_3$ and $d$ with $\frac{\max\{r,\kappa q\}+p^*}{2}$, $\frac{\max\{s,q\}+p^*}{2}$, $2M_3$ and $d+1$, respectively if necessary, we may assume that
\begin{equation}\label{SE}
q \ll r \ll p^*,\ \ q\ll s \ll q^*,\ \ \text{and}\ \ \kappa \ll \frac{r}{q}.
\end{equation}
We proceed in some steps. 

\bigskip
{\bf Step 1: Localizing.}${}$

\vskip5pt
\noindent  By \eqref{SE}, it holds that
\begin{align} \label{def_delta}
	\begin{split}
		\delta:=\frac{1}{2} \min \left\{ \left(r-p\right)^-,\left(p^\ast-r\right)^-,  \left(s-q\right)^-, \left(q^\ast-s\right)^-,\left(\frac{r}{q}-\kappa\right)^-   \right\}>0.
	\end{split}
\end{align}
Then, by the uniform continuity of exponents, we can fix $\eps \in \left(0,1\right)$ small enough such that for all $x_0 \in \RN$, it holds  
\begin{equation}\label{S-D.loc.exp.1}
	\delta<\min\left\{r_{x_0,\eps}^--p_{x_0,\eps}^-, \left(p_{x_0,\eps}^-\right)^\ast-r_{x_0,\eps}^+, s_{x_0,\eps}^--q_{x_0,\eps}^+, \left(q_{x_0,\eps}^-\right)^\ast-s_{x_0,\eps}^+, \frac{r_{x_0,\eps}^-}{q_{x_0,\eps}^+}-\kappa_{x_0,\eps}^+
	\right\}.
\end{equation}
Let $x_0 \in \RN$ be fixed. For simplicity, we shall henceforth write $Q_\eta$ instead of $Q_\eta(x_0)$ for each $\eta>0$, and for $h\in C(\RN)$, we shall write $h_{\eps}^+$ and $h_{\eps}^-$ in place of $h_{x_0,\eps}^+$ and $h_{x_0,\eps}^-$, respectively. We aim to show that
\begin{equation}\label{PTS.priori-l}
	\|u\|_{L^\infty(Q_{\frac{\eps}{3})}} \leq C \max \left\{\Big(\int_{ Q_{\eps}}\mathcal{F}(x,|u|) \diff x\Big)^{\bar{\tau}_1}, \Big(\int_{ Q_{\eps}} \mathcal{F}(x,|u|) \diff x\Big)^{\bar{\tau}_2}\right\},
\end{equation}
where $C$, $\bar{\tau}_1$ and $\bar{\tau}_2$ are positive constants independent of $u$ and $x_0$.

\bigskip

{\bf Step 2: Defining the recursion sequence and basic estimates.}${}$

\vskip5pt
\noindent We define three sequences $\{\eps_n\}_{n\in\N_0}$, $\{\bar{\eps}_n\}_{n\in\N_0}$ and $\{k_n\}_{n\in\N_0}$ as follows:
\begin{equation*}
	\eps_n:= \frac{\eps}{3} \left(1 +\frac{1}{2^{n-1}}\right), \ \bar{\eps}_n:=\frac{\eps_n+\eps_{n+1}}{2} \quad \text{for} \ n\in \mathbb{N}_0,
\end{equation*} 
and
\begin{align}\label{def.kn}
	k_n:=k_*\left(2-\frac{1}{2^n}\right) \quad \text{for} \ n\in \mathbb{N}_0
\end{align}
with $k_*>0$ being specified later. 
Clearly, we have
$$\eps_n \downarrow \frac{\eps}{3} \quad \text{and}\quad  \frac{\eps}{3} <\eps_{n+1}<\bar{\eps}_n<\eps_n \leq \eps   \quad \text{for all} \ n\in \mathbb{N}_0$$
and
$$k_n \uparrow 2k_* \quad \text{and}\quad k_* \leq k_n<k_{n+1} <2k_*\quad \text{for all} \ n\in\ \mathbb{N}_0.$$
For each $n\in\N_0$, we define
\begin{equation}\label{Dn}
	D_n:=A_{k_n} \cap Q_{\eps_n}\quad \text{and}\quad \Lambda_n:= A_{k_{n+1}} \cap Q_{\bar{\eps}_n},
\end{equation}
where for each $k\in\R$, the level set $A_k$ is defined by
\begin{align}\label{def.Ak}
	A_k:=\{x\in\RN\,:\,\ u(x)>k\}.
\end{align}
 Thus, we have
\begin{equation}\label{T.Sub1.C}
	Q_{\frac{\eps}{3}} \subset Q_{\eps_{n+1}} \subset Q_{\bar{\eps}_n} \subset Q_{\eps_n} \subset Q_{\eps}\quad\text{and}\quad D_{n+1} \subset \Lambda_n \subset D_n \quad \text{for all} \ n \in \N_0.
\end{equation}
We now define the recursion sequence $\{Y_n\}_{n\in\N_0}$ as follows
\begin{equation}\label{Yn}
	Y_n:= \int_{D_n}  \mathcal{F}\left(x,u-k_n\right) \diff x \quad \text{for all} \ n \in \N_0,
\end{equation}
where $\mathcal{F}$ is given in \eqref{def_F}. Clearly, we have
\begin{align*}
	\quad Y_{n+1} \leq Y_n \quad \text{for all} \ n \in \N_0.
\end{align*}
For the remainder of the proof, let $C_j$ ($j\in\N$) represent a positive constant independent of $u$, $x_0$, $n$, and $k_\ast$. From the estimate
\begin{align*}
	u(x)- k_n \geq u(x)\l(1-\frac{k_n}{k_{n+1}}\r) = \frac{u(x)}{2^{n+2}-1} \quad \text{for a.a.\,} x \in A_{k_{n+1}},
\end{align*}
we easily obtain
\begin{align}\label{def.Fn}
	F_n:=\int_{\Lambda_n} \mathcal{F}(x,u) \diff x \leq C_1 2^{ns^+}Y_n, \quad \forall n \in \N_0.
\end{align}
By \eqref{def.kn} and \eqref{T.Sub1.C} it follows that
\begin{align*}
	|\Lambda_n|
	\leq \int_{\Lambda_n} \left (\frac{u-k_n}{k_{n+1}-k_n}\right )^{r(x)}  \diff x
	& \leq \int_{D_{n}} \frac{2^{r(x)(n+1)}}{k_*^{r(x)}} (u-k_n)^{r(x)} \diff x\\
	& \leq \int_{D_{n}} \frac{2^{r(x)(n+1)}}{k_*^{r(x)}} \mathcal{F}(x,u-k_n) \diff x,
\end{align*}
and thus,
\begin{align} \label{S-|A_{k_{n+1}}|}
	|D_{n+1}| \leq 	|\Lambda_n| \leq  C_2 \left(k_{*}^{-r^{-}}+k_{*}^{-r^{+}}\right)2^{nr^+} Y_{n}\leq 2C_2 \left(1+k_{*}^{-r^{+}}\right)2^{nr^+} Y_{n}, \quad \forall n \in \N_0.
\end{align}
For each $n \in \N_0$, let $v_n:=(u-k_{n+1})_+$ and let $\xi_n \in C^\infty(\RN)$ be such that 
\begin{align}\label{def.xin}
\supp (\xi_n) \subset Q_{\bar{\eps}_n}, ~ \xi_n(\cdot)=1 ~ \text{on}~ Q_{\eps_{n+1}},~	0 \leq \xi_n(\cdot) \leq 1~  \text{and}~ |\nabla \xi_n(\cdot)|<C_32^n ~  \text{on}~\R^N.
\end{align}
 Denote
\begin{equation} \label{def.Rn}
	R_n: =\int_{\Lambda_n} \big[\mathcal{H}(x,|\nabla v_n|) + V(x) \mathcal{H}(x,v_n)\big]\xi_n^{q_\eps^+}  \diff x.
\end{equation}
We claim that
\begin{equation}\label{est.R}
	\begin{split}
		R_n  \leq C_4 2^{nq^+}\left(1 + k^{-r^+}_\ast\right) \left( F_n + F_n^{\frac{1}{\kappa_\eps^+}} + |\Lambda_n| \right).
	\end{split} 
\end{equation}
To this end, taking  $\varphi_n:=v_n\xi_n^{q_\eps^+}$ as a test function in  \eqref{def.weak.sol}, and  noting that $\textup{supp} (\varphi_n) \subset \Lambda_n$ and $\nabla u=\nabla v_n$ on $\Lambda_n$, we obtain
\begin{multline}\label{S.var.Eq}
\int_{\Lambda_n}
\left(\mathcal{A}(x,u,\nabla v_n) \cdot \nabla v_n\right)\xi_n^{q_\eps^+}  \diff x +  \int_{\Lambda_n}V(x) \Big(|u|^{p(x)-2}+a(x)|u|^{q(x)-2}\Big)uv_n\xi_n^{q_\eps^+} \diff x   \\
=\int_{\Lambda_n} \mathcal{B}(x,u,\nabla v_n)v_n\xi_n^{q_\eps^+} \diff x 
- \int_{\Lambda_n}	\left( \mathcal{A}(x,u,\nabla v_n) \cdot \nabla \xi_n \right) v_n  q_\eps^+ \xi_n^{q_\eps^+-1} \, \diff x.
\end{multline}
It yields from  assumption \eqref{H1} (ii) that
\begin{multline*} 
	\int_{\Lambda_n}\left(\mathcal{A}(x,u,\nabla v_n) \cdot \nabla v_n \right) \xi_n^{q_\eps^+}  \diff x  \\
	\geq M_2\int_{\Lambda_n}\left[|\nabla v_n|^{p(x)}+a(x)|\nabla v_n|^{q(x)}\right] \xi_n^{q_\eps^+} \diff x - M_3 \int_{\Lambda_n} \left[u^{r(x)}+a(x)^{\frac{s(x)}{q(x)}}u^{s(x)}+1\right]\xi_n^{q_\eps^+} \diff x.
\end{multline*}
Combining this and the fact that $u > v_n>0$ on $\Lambda_n$ gives
\begin{multline}\label{A.1}
\int_{\Lambda_n} \left(\mathcal{A}(x,u,\nabla v_n) \cdot \nabla v_n \right) \xi_n^{q_\eps^+}  \diff x + \int_{\Lambda_n}V(x) \Big(|u|^{p(x)-2}+a(x)|u|^{q(x)-2}\Big)uv_n\xi_n^{q_\eps^+} \diff x   \\
\geq M_*  R_n - M_3 \int_{\Lambda_n}  \mathcal{F}(x,u) \diff x - M_3 |\Lambda_n|,\quad \text{with} \ M_*:=\min\{1,M_2\}.
\end{multline}
\noindent On the other hand, using \eqref{H1} (iii) and the fact that $0<v_n<u$ on $\Lambda_n$ again we have
\begin{multline}\label{B.1}
\left|\int_{\Lambda_n}  \mathcal{B}(x,u,\nabla v_n)v_n\xi_n^{q_\eps^+} \diff x \right| \\
\leq L\int_{\Lambda_n} \left( |d(x)|u + \mathcal{F}(x,u) +|\nabla v_n|^{\frac{p(x)}{r'(x)}}u + a(x)^{\frac{1}{q(x)}+\frac{1}{s'(x)}}|\nabla v_n|^{\frac{q(x)}{s'(x)}}u\right)\xi_n^{q_\eps^+} \diff x.
\end{multline}
Applying Young's inequality gives
\begin{align*}
& \int_{\Lambda_n} \left( |\nabla v_n|^{\frac{p(x)}{r'(x)}} + a(x)^{\frac{1}{q(x)}+\frac{1}{s'(x)}}|\nabla v_n|^{\frac{q(x)}{s'(x)}} \right) u\xi_n^{q_\eps^+}  \diff x \\
& \quad \leq \frac{M_*}{4L}  \int_{\Lambda_n} \left( |\nabla v_n|^{p(x)}+ a(x) |\nabla v_n|^{q(x)} \right) \xi_n^{q_\eps^+} \, \diff x 
	+ C_5 \int_{\Lambda_n} \left( u^{r(x)} + a(x) ^{\frac{s(x)}{q(x)}}u^{s(x)} \right)\xi_n^{q_\eps^+} \diff x\\
& \quad	\leq \frac{M_*}{4L} \int_{\Lambda_n} \mathcal{H}(x,|\nabla v_n|)\xi_n^{q_\eps^+} \diff x + C_5 \int_{\Lambda_n} \mathcal{F}(x,u)\diff x.
\end{align*}
Combining the last estimate with \eqref{B.1} yields
\begin{multline}\label{B.2}
\left|\int_{\Lambda_n}  \mathcal{B}(x,u,\nabla v_n)v_n\xi_n^{q_\eps^+} \diff x \right|\leq \frac{M_*}{4}\int_{\Lambda_n} \mathcal{H}(x,|\nabla v_n|)\xi_n^{q_\eps^+} \diff x \\
+ C_6 \left[\int_{\Lambda_n}\mathcal{F}(x,u) \diff x +  \int_{\Lambda_n} |d(x)|u \diff x\right].
\end{multline}
In order to estimate the last term of \eqref{S.var.Eq}, we make use of \eqref{H1} (i) and the fact that $0<v_n<u$ on $\Lambda_n$ again to obtain
\begin{align}\label{A.2}
\notag	\Big| \int_{\Lambda_n} & \left( \mathcal{A}(x,u,\nabla v_n) \cdot \nabla \xi_n \right) v_n  q_\eps^+ \xi_n^{q_\eps^+-1} \, \diff x \Big|\\
\notag	 \leq& M_1q_\eps^+ \int_{\Lambda_n} \left[|\nabla v_n|^{p(x)-1} +a(x)|\nabla v_n|^{q(x)-1}\right]|\nabla \xi_n|v_n\xi_n^{q_\eps^+-1} \, \diff x \\ 
&+  M_1q_\eps^+\int_{\Lambda_n} \left[ u^{r(x)} + a(x) ^{\frac{s(x)}{q(x)}}u^{s(x)}\right]|\nabla \xi_n|\xi_n^{q_\eps^+-1}\, \diff x + M_1q_\eps^+ \int_{\Lambda_n} |k(x)||\nabla \xi_n|v_n\xi_n^{q_\eps^+-1} \diff x.
\end{align}
Next, we shall estimate each term on the right-hand side of \eqref{A.2}. 
By Young's inequality, \eqref{def.xin} and the fact that $\frac{(q_\eps^+-1)p(x)}{p(x)-1}\geq q_\eps^+$ for all $x\in\Lambda_n$ it holds that
\begin{align}\label{S.est.2'}
	|\nabla v_n|^{p(x)-1} |\nabla \xi_n|v_n\xi_n^{q_\eps^+-1}	& \leq \frac{M_*}{4M_1q_\eps^+} |\nabla v_n|^{p(x)}\xi_n^{\frac{(q_\eps^+-1)p(x)}{p(x)-1}}+C_7|\nabla \xi_n|^{p(x)}v_n^{p(x)} \notag \\
	& \leq \frac{M_*}{4M_1q_\eps^+} |\nabla v_n|^{p(x)}\xi_n^{q_\eps^+}+C_7|\nabla \xi_n|^{p(x)}v_n^{p(x)}\notag\\
	& \leq \frac{M_*}{4M_1q_\eps^+} |\nabla v_n|^{p(x)}\xi_n^{q_\eps^+}+C_82^{np^+}v_n^{p(x)}.
\end{align}
Similarly, it holds that
\begin{align}\label{S.est.2''}
	|\nabla v_n|^{q(x)-1} |\nabla \xi_n|v_n\xi_n^{q_\eps^+-1}	& \leq \frac{M_*}{4M_1q_\eps^+} |\nabla v_n|^{q(x)}\xi_n^{q_\eps^+}+C_92^{nq^+}v_n^{q(x)} .
\end{align}
From  \eqref{S.est.2'}-\eqref{S.est.2''}, we obtain
\begin{align}\label{S.est.2}
	\int_{\Lambda_n} & \left[|\nabla v_n|^{p(x)-1} +a(x)|\nabla v_n|^{q(x)-1}\right]|\nabla \xi_n|v_n\xi_n^{q_\eps^+-1}\, \diff x  \notag \\
	& \leq \frac{M_*}{4M_1q_\eps^+} \int_{\Lambda_n} \mathcal{H}(x,|\nabla v_n|)\xi_n^{q_\eps^+}\diff x
	+ C_{10}2^{nq^+} \int_{\Lambda_n} \mathcal{H}(x,v_n)  \diff x \notag \\
	& \leq \frac{M_*}{4M_1q_\eps^+} \int_{\Lambda_n} \mathcal{H}(x,|\nabla v_n|)\xi_n^{q_\eps^+} \diff x + C_{11} 2^{nq^+}  \left[\int_{\Lambda_n} \mathcal{F}(x,u) \diff x +  |\Lambda_n|\right].
\end{align}
Here, we have used the basic estimates: $v_n^{p(x)}\leq u^{p(x)}\leq 1+u^{r(x)}$ and $a(x)v_n^{q(x)}\leq a(x)u^{q(x)}\leq 1+a(x)^{\frac{s(x)}{q(x)}}u^{s(x)}$ for a.a. $x\in \Lambda_n$. Clearly, it holds that
\begin{align}\label{S.est.1}
	M_1q_\eps^+\int_{\Lambda_n} \left[ u^{r(x)} + a(x) ^{\frac{s(x)}{q(x)}}u^{s(x)}\right]|\nabla \xi_n|\xi_n^{q_\eps^+-1}\, \diff x &  \leq  C_{12} 2^{n} \int_{\Lambda_n} \mathcal{F}(x,u) \diff x.
\end{align}
Gathering \eqref{A.2}, \eqref{S.est.2} and \eqref{S.est.1}, we obtain
\begin{align}\label{S.A}
\notag	\Big| \int_{\Lambda_n} & \left( \mathcal{A}(x,u,\nabla v_n) \cdot \nabla \xi_n \right) v_n  q_\eps^+ \xi_n^{q_\eps^+-1} \, \diff x \Big| \\
	&\leq \frac{M_*}{4} R_n + C_{13}2^{nq^+} \left[ \int_{\Lambda_n} \mathcal{F}(x,u) \diff x +  |\Lambda_n| +  \int_{\Lambda_n} |k(x)|v_n \diff x \right].
\end{align}
Utilizing \eqref{A.1}, \eqref{B.2} and \eqref{S.A}, we derive from \eqref{S.var.Eq} that
\begin{align} \label{local.est.1}
R_n\leq C_{14} 2^{nq^+} \left[ \int_{\Lambda_n} \mathcal{F}(x,u) \diff x +  \int_{\Lambda_n} |d(x)|u \diff x + \int_{\Lambda_n} |k(x)|u \diff x +  |\Lambda_n| \right].
\end{align}
On the other hand,  by the virtue of Propositions~\ref{prop.Holder}-\ref{DHHR} and \eqref{k-d.1} it holds that
\begin{align}\label{es.d(x)}
\notag\int_{\Lambda_n} |d(x)|u \diff x &	\leq 2\|d\|_{L^{\sigma'(\cdot)}(\Lambda_n)}\|u\|_{L^{\sigma(\cdot)}(\Lambda_n)}\\ \notag& \le 4\|d\|_{L^{\sigma'(\cdot)}(\Lambda_n)}\|u\|_{L^{\kappa_\eps^+}(\Lambda_n)} \\
\notag& \leq C_{14} \left(\int_{\Lambda_n} u^{\kappa_\eps^+} \diff x \right)^{\frac{1}{\kappa_\eps^+}} =  C_{14} \left(\int_{\Lambda_n} u^{r(x)}u^{\kappa_\eps^+-r(x)} \diff x \right)^{\frac{1}{\kappa_\eps^+}} \\
\notag& \leq C_{15}\left(1+k_\ast^{\frac{\kappa_\eps^+-r_\eps^+}{\kappa_\eps^+}}\right)\left(\int_{\Lambda_n} u^{r(x)}\diff x \right)^{\frac{1}{\kappa_\eps^+}} \\
& \leq C_{15}\left(1+k_\ast^{\frac{\kappa_\eps^+- r_\eps^+}{\kappa_\eps^+}}\right)\left(\int_{\Lambda_n} \mathcal{F}(x,u)\diff x \right)^{\frac{1}{\kappa_\eps^+}}.
\end{align}
Similarly, 
\begin{equation} \label{es.k(x)}
	\begin{split}
		\int_{\Lambda_n} |k(x)|u \diff x 
		& \leq C_{16}\left(1+k_\ast^{\frac{\kappa_\eps^+- r_\eps^+}{\kappa_\eps^+}}\right)\left(\int_{\Lambda_n} \mathcal{F}(x,u)\diff x \right)^{\frac{1}{\kappa_\eps^+}}.
	\end{split}
\end{equation} 
Referring \eqref{es.d(x)} and \eqref{es.k(x)}, we deduce \eqref{est.R} from \eqref{local.est.1}. 

\bigskip

{\bf Step 3: Establishing the recursion inequalities.}${}$

\vskip5pt
\noindent  With $\{Y_n\}_{n\in\N_0}$ defined in \eqref{Yn}, we aim to show  that
\begin{equation} \label{Yn+1.0}
Y_{n+1} \leq \bar{C} \left(k_\ast^{-\alpha_1} + k_\ast^{-\alpha_2} \right)b^n \left(Y_n^{1+\mu_1} + Y_n^{1+\mu_2}\right),\quad\forall n\in\N_0,
\end{equation} 
where $\bar{C}$, $b>1$, $\alpha_1$, $\alpha_2$, $\mu_1$ and $\mu_2$ are positive constants independent of $u$, $x_0$, $k_\ast$, and $n$. Let $n\in\N_0$, and let $\varphi_n$ be as in Step 2. For $\hat{\alpha},\hat{\beta}>0$ and $x\in\RN$, denote
\begin{align} \label{def.F.2}
	\wt{F}(\hat{\alpha},\hat{\beta},\varphi_n(x)):=\varphi_n(x)^{\hat{\alpha}}+a(x)^{\frac{\hat{\beta}}{q(x)}}\varphi_n(x)^{\hat{\beta}}.
\end{align} 	
  Since $t^{h(x)} \leq t^{h^+} + t^{h^-}$ for $(x,t) \in \RN\times [0,\infty)$ and $h \in C_+(\RN)$, one has
\begin{align} \label{Yn+1.1}
	Y_{n+1} & = \int_{D_{n+1}} \mathcal{F}(x,u-k_{n+1}) \diff x = \int_{D_{n+1}} \mathcal{F}\left(x,\varphi_n\right) \diff x \notag \\
	& \leq \int_{D_{n+1}} \wt{F} \left(r_\eps^-,s_\eps^-, \varphi_n(x)\right)  \diff x + \int_{D_{n+1}} \wt{F} \left(r_\eps^+,s_\eps^+, \varphi_n(x)\right)  \diff x.
\end{align}
From \eqref{S-D.loc.exp.1}, one has
\begin{align}\label{S-D-eps}
	0<\de< \min\left\{(p^*)_\eps^--r_\eps^+, (q^*)_\eps^--s_\eps^+ \right\}.
\end{align}
For $\star \in \{+,-\}$, invoking H\"older's inequality and noticing $|D_{n+1}|<1$, we obtain
\begin{align}\label{S-de-E1}
\notag	\int_{D_{n+1}} &\wt{F} \left(r_\eps^\star,s_\eps^\star, \varphi_n(x) \right)  \diff x\\
 \notag	&	\leq \left(\int_{D_{n+1}}\varphi_n^{r_\eps^\star+\de} \diff x\right)^{\frac{r_\eps^\star}{r_\eps^\star+\de}}|D_{n+1}|^{\frac{\de}{r_\eps^\star+\de}}+ \left(\int_{D_{n+1}} a(x)^{\frac{s_\eps^\star+\de}{q(x)}}\varphi_n^{s_\eps^\star+\de} \diff x\right)^{\frac{s_\eps^\star}{s_\eps^\star+\de}}|D_{n+1}|^{\frac{\de}{s_\eps^\star+\de}} \\
\notag	&\leq |D_{n+1}|^{\frac{\de}{r^++s^++\de}}\left[\left(\int_{Q_{\eps}}\varphi_n^{r_\eps^\star+\de} \diff x\right)^{\frac{r_\eps^\star}{r_\eps^\star+\de}}+\left(\int_{Q_{\eps}}a(x)^{\frac{s_\eps^\star+\de}{q(x)}}\varphi_n^{s_\eps^\star+\de} \diff x\right)^{\frac{s_\eps^\star}{s_\eps^\star+\de}}\right]\\
	&\leq |D_{n+1}|^{\frac{\de}{r^++s^++\de}}\left(\|\varphi_n\|_{L^{\mathcal{B}^\star}(Q_{\eps})}^{r_\eps^\star}+\|\varphi_n\|_{L^{\mathcal{B}^\star}(Q_{\eps})}^{s_\eps^\star}\right),	 
\end{align}
where
\begin{align*}
	\mathcal{B}^\star(x,t):=t^{r_\eps^\star+\de}+a(x)^{\frac{s_\eps^\star+\de}{q(x)}}t^{s_\eps^\star+\de} ~\text{for}~ x \in Q_{\eps}~ \text{and}~ t \in [0,+\infty).
\end{align*}
Note that $r_\eps^\star+\de<(p^*)_\eps^-$ and $s_\eps^\star+\de  < (q^*)_\eps^-$ due to \eqref{S-D-eps}. Then, by applying Proposition~\ref{RmkEm} and then Proposition~\ref{P.E2} with $\mathcal{F}=\mathcal{B}^\star$ we obtain
\begin{align}  \label{S-de-E2*}
	 \|\varphi_n\|_{L^{\mathcal{B}^\star}(Q_{\eps})}
	\leq C_{17}\|\varphi_n\|_{W^{1,\mathcal{H}}(Q_{\eps})}=C_{17}\|\varphi_n\|_{W^{1,\mathcal{H}}(\RN)}\leq C_{18}\|\varphi_n\|.
\end{align}
Putting
\begin{equation*}
	\rho_n :=	\int_{\R^N}  \Big[\mathcal{H} (x, |\nabla \varphi_n|) + V(x) \mathcal{H}(x,\varphi_n)\Big] \diff x,
\end{equation*}
we deduce from \eqref{S-de-E2*} and \eqref{m-n} that
\begin{align*}
	\|\varphi_n\|_{L^{\mathcal{B}^\star}(Q_{\eps})}^{r_\eps^\star}\leq (1+C_{18}^{r^+})\left( \rho_n^{\frac{r^{\star}_\eps}{p_\eps^-}} +  \rho_n^{\frac{r^{\star}_\eps}{q_\eps^+}} \right).
\end{align*}
Similarly, we have
\begin{align*}
	\|\varphi_n\|_{L^{\mathcal{B}^\star}(Q_{\eps})}^{s_\eps^\star}\leq (1+C_{18}^{s^+})\left( \rho_n^{\frac{s^{\star}_\eps}{p_\eps^-}} +  \rho_n^{\frac{s^{\star}_\eps}{q_\eps^+}} \right).
\end{align*}
From the last two estimates, we easily obtain
\begin{align*}
\|\varphi_n\|_{L^{\mathcal{B}^\star}(Q_{\eps})}^{r_\eps^\star}+	\|\varphi_n\|_{L^{\mathcal{B}^\star}(Q_{\eps})}^{s_\eps^\star}\leq C_{19}\left( \rho_n^{\frac{s^+_\eps}{p_\eps^-}} +  \rho_n^{\frac{r^-_\eps}{q_\eps^+}} \right).
\end{align*}
On the other hand, recalling the definitions \eqref{def.xin}, \eqref{def.Rn} and \eqref{def.Fn} for $\xi_n$, $R_n$ and $F_n$, we obtain
\begin{equation*}
\begin{split}
	\rho_n &= \int_{\Lambda_n}  \left[\mathcal{H} \left(x, \left|\xi_n^{q_\eps^+}\nabla v_n + v_n \nabla (\xi_n^{q_\eps^+})\right|\right) + V(x) \mathcal{H}(x,\varphi_n)\right] \diff x \\
	&	\leq C_{20} \int_{\Lambda_n} \big[ \mathcal{H}(x,|\nabla v_n|) + V(x) \mathcal{H}(x,v_n)\big] \xi_n^{q_\eps^+} \diff x + C_{20}2^{nq^+} \int_{\Lambda_n}  \mathcal{H}(x,v_n) \diff x \\
	&	\leq C_{21} 2^{nq^+} \Big(R_n +  F_n + |\Lambda_n|\Big).
\end{split}
\end{equation*}
Combining the last two estimates with the fact that $|\Lambda_n|<1$ we obtain
\begin{align*}
\|\varphi_n\|_{L^{\mathcal{B}^\star}(Q_{\eps})}^{r_\eps^\star}+	\|\varphi_n\|_{L^{\mathcal{B}^\star}(Q_{\eps})}^{s_\eps^\star} & \leq C_{22}2^{\frac{q^+s^+n}{p^-}} \left(R_n^{\frac{s^+_\eps}{p_\eps^-}} +  R_n^{\frac{r^-_\eps}{q_\eps^+}} +F_n^{\frac{s^+_\eps}{p_\eps^-}} +  F_n^{\frac{r^-_\eps}{q_\eps^+}}+ |\Lambda_n|^{\frac{r^-_\eps}{q_\eps^+}} \right).
\end{align*}
This and \eqref{est.R} imply
\begin{align} \label{phi.r}
\|\varphi_n\|_{L^{\mathcal{B}^\star}(Q_{\eps})}^{r_\eps^\star}+	\|\varphi_n\|_{L^{\mathcal{B}^\star}(Q_{\eps})}^{s_\eps^\star} \leq C_{23}2^{\frac{2q^+s^+}{p^-}n} \left(1+k_\ast^{-r^+}\right)^{\frac{s_\eps^+}{p_\eps^-}} \left(F_n^{\frac{r^-_\eps}{\kappa_\eps^+ q_\eps^+}} + F_n^{\frac{s^+_\eps}{ p_\eps^-}} + |\Lambda_n|^{\frac{r^-_\eps}{q_\eps^+}} \right).
\end{align}
Utilizing \eqref{S-de-E1} and \eqref{phi.r}, we obtain
\begin{align} \label{Yn+1.2}
	\int_{D_{n+1}} \wt{F} \left(r_\eps^\star,s_\eps^\star, \varphi_n(x) \right)  \diff x \leq C_{24} b_0^n |D_{n+1}|^{\ell_\de} \left(1+k_\ast^{-r^+}\right)^{\frac{s_\eps^+}{p_\eps^-}} \left(F_n^{\frac{r^-_\eps}{\kappa_\eps^+ q_\eps^+}} + F_n^{\frac{s^+_\eps}{ p_\eps^-}} + |\Lambda_n|^{\frac{r^-_\eps}{q_\eps^+}} \right),
\end{align}
where
\begin{equation}\label{l.del}
	 b_0:= 2^{\frac{2q^+s^+}{p^-}}>1 \quad \text{and}\  \ell_\de:=\frac{\de}{r^+ + s^+ + \de}.
\end{equation}
It follows from  \eqref{Yn+1.1} and \eqref{Yn+1.2} that
\begin{align*} \label{Yn+1.3}
	Y_{n+1} \leq C_{25} b_0^n |D_{n+1}|^{\ell_\de} \left(1+k_\ast^{-r^+}\right)^{\frac{s_\eps^+}{p_\eps^-}} \left(F_n^{\frac{r^-_\eps}{\kappa_\eps^+ q_\eps^+}} + F_n^{\frac{s^+_\eps}{ p_\eps^-}} + |\Lambda_n|^{\frac{r^-_\eps}{q_\eps^+}} \right).
\end{align*}
Utilizing \eqref{def.Fn} and \eqref{S-|A_{k_{n+1}}|}, the last estimate yields
\begin{multline*}
	Y_{n+1} \leq C_{25} b_0^n \left[C_2 \left(k_{*}^{-r^{-}}+k_{*}^{-r^{+}}\right)2^{nr^+} Y_{n}\right]^{\ell_\delta} \left(1+k_\ast^{-r^+}\right)^{\frac{s_\eps^+}{p_\eps^-}}\times\\
	\times\left[\left(C_1 2^{ns^+}Y_n\right)^{\frac{r^-_\eps}{\kappa_\eps^+ q_\eps^+}}+ \left(C_1 2^{ns^+}Y_n\right)^{\frac{s^+_\eps}{ p_\eps^-}}+\left(2C_2 \left(1+k_{*}^{-r^{+}}\right)2^{nr^+} Y_{n}\right)^{\frac{r^-_\eps}{q_\eps^+}}  \right]. 
\end{multline*}
This leads to
\begin{align*} 
	Y_{n+1} &\leq C_{26} b_0^n \left(k_\ast^{-r^-\ell_\delta} + k_\ast^{-r^+\ell_\delta} \right) \left(1+k_\ast^{-\frac{r^+s^+}{p^-}}\right)2^{\frac{2(s^+)^2}{p^-}n} \left(Y_n^{\frac{r^-_\eps}{\kappa_\eps^+ q_\eps^+}} + Y_n^{\frac{s^+_\eps}{ p_\eps^-}}  \right) \notag\\
	&\leq C_{27}  \left(k_\ast^{-r^-\ell_\delta} + k_\ast^{-r^+\left(\ell_\delta+\frac{s^+}{p^-}\right)} \right)\left(2^{\frac{2(s^+)^2}{p^-}}b_0\right)^n \left(Y_n^{\frac{r^-_\eps}{\kappa_\eps^+ q_\eps^+}} + Y_n^{\frac{s^+_\eps}{ p_\eps^-}}  \right). 
\end{align*}
Thus, we arrive at
\begin{align*} 
	Y_{n+1} \leq C_{27}  \left(k_\ast^{-\alpha_1} + k_\ast^{-\alpha_2} \right)b^n \left(Y_n^{1+\mu_1} + Y_n^{1+\mu_2}\right),
\end{align*}
where
\begin{gather*}
	0 < \alpha_1:=r^-\ell_\de < \alpha_2:= r^+ \left(\ell_\de + \frac{s^+}{p^-}\right),\ \
	1< b:= 2^{\frac{2(s^+)^2}{p^-}}b_0, \\
\end{gather*}
and 
\begin{equation*}
	0 < \frac{\delta}{\kappa^+}\leq \mu_1:= \inf_{x_0\in\R^N}\frac{r^-_{x_0,\eps}}{\kappa_{x_0,\eps}^+ q_{x_0,\eps}^+} -1 \leq \mu_2:=\sup_{x_0\in\R^N}\frac{s^+_{x_0,\eps}}{ p_{x_0,\eps}^-} -1
\end{equation*}
(see \eqref{S-D.loc.exp.1}). Thus, we establish the recursion sequence \eqref{Yn+1.0} with $\bar{C}=C_{27}+1$.

\bigskip

{\bf Step 4: Establishing a priori bounds.}${}$

\vskip5pt
\noindent We shall show \eqref{PTS.priori-l} by applying Lemma \ref{leRecur} for $\{Y_n\}_{n\in\N_0}$. Clearly, \eqref{PTS.priori-l} holds when $\int_{ Q_{\eps}} \mathcal{F}\left(x, |u|\right) \diff x=0$. We now consider the case $\int_{ Q_{\eps}} \mathcal{F}\left(x, |u|\right) \diff x>0$. Let $\bar{K}:=k_\ast^{-\alpha_1} + k_\ast^{-\alpha_2}$ for simplicity.  By Lemma \ref{leRecur}, it holds that
\begin{equation} \label{limYn.0}
	Y_n \to 0 \quad \text{as} \ n \to \infty
\end{equation}
provided 
\begin{equation} \label{Y0}
	\begin{split} 
			Y_0 &= \int_{A_{k_\ast} \cap Q_{\eps}} \mathcal{F}(x,u-k_\ast) \diff x  \\
		& \leq \min\left\{(2\bar{C}\bar{K})^{-\frac{1}{\mu_1}}b^{-\frac{1}{\mu_1^2}}, \round{2\bar{C}\bar{K}}^{-\frac{1}{\mu_2}}b^{-\frac{1}{\mu_1\mu_2}-\frac{\mu_2-\mu_1}{\mu_2^2}}\right\}.
	\end{split}
\end{equation} 
Note that 
\begin{equation*}
	Y_0 = \int_{ Q_{\eps}} \mathcal{F}\left(x, (u-k_\ast)_+\right) \diff x \leq \int_{ Q_{\eps}} \mathcal{F}\left(x, |u|\right) \diff x.
\end{equation*}
Hence, \eqref{Y0} is satisfied if $k_\ast$ fulfills 
\begin{equation} \label{k*.1}
	\int_{ Q_{\eps}} \mathcal{F}\left(x, |u|\right) \diff x \leq \min\left\{(2\bar{C}\bar{K})^{-\frac{1}{\mu_1}}b^{-\frac{1}{\mu_1^2}},\  \round{2\bar{C}\bar{K}}^{-\frac{1}{\mu_2}}b^{-\frac{1}{\mu_1\mu_2}-\frac{\mu_2-\mu_1}{\mu_2^2}}\right\},
\end{equation}
which is equivalent to 
\begin{equation} \label{k*.1'}
	\begin{cases}
	\bar{K}
	\leq (2\bar{C})^{-1}b^{-\frac{1}{\mu_1}}\Big(	\int_{ Q_{\eps}}\mathcal{F}(x,|u|) \diff x\Big)^{-\mu_1},\\
	\bar{K} 
	  \leq
	(2\bar{C})^{-1}b^{-\frac{1}{\mu_1}-\frac{\mu_2-\mu_1}{\mu_2}}\Big(\int_{ Q_{\eps}} \mathcal{F}(x,|u|) \diff x\Big)^{-\mu_2}.
	\end{cases}
\end{equation}
Meanwhile, \eqref{k*.1'} is fulfilled if it holds that
\begin{equation*}
	\begin{cases}
	2k_\ast^{-\alpha_1}
	\leq (2\bar{C})^{-1}b^{-\frac{1}{\mu_1}-\frac{\mu_2-\mu_1}{\mu_2}} \min \left\{\Big(	\int_{ Q_{\eps}}\mathcal{F}(x,|u|) \diff x\Big)^{-\mu_1}, \Big(	\int_{ Q_{\eps}} \mathcal{F}(x,|u|) \diff x\Big)^{-\mu_2}\right\},\\
	2k_\ast^{-\alpha_2}
	\leq (2\bar{C})^{-1}b^{-\frac{1}{\mu_1}-\frac{\mu_2-\mu_1}{\mu_2}} \min \left\{\Big(	\int_{ Q_{\eps}}\mathcal{F}(x,|u|) \diff x\Big)^{-\mu_1}, \Big(	\int_{ Q_{\eps}} \mathcal{F}(x,|u|) \diff x\Big)^{-\mu_2}\right\},
\end{cases}
\end{equation*}
which is equivalent to
\begin{equation*}
	\begin{cases}
	k_\ast
	\geq( 4\bar{C})^{\frac{1}{\alpha_1}} b^{\frac{1}{\alpha_1} \left(\frac{1}{\mu_1}+\frac{\mu_2-\mu_1}{\mu_2}\right)} \max \left\{\Big(	\int_{ Q_{\eps}}\mathcal{F}(x,|u|) \diff x\Big)^{\frac{\mu_1}{\alpha_1}}, \Big(	\int_{ Q_{\eps}} \mathcal{F}(x,|u|) \diff x\Big)^{\frac{\mu_2}{\alpha_1}}\right\},\\
	k_\ast
	\geq( 4\bar{C})^{\frac{1}{\alpha_2}} b^{\frac{1}{\alpha_2} \left(\frac{1}{\mu_1}+\frac{\mu_2-\mu_1}{\mu_2}\right)} \max \left\{\Big(	\int_{ Q_{\eps}}\mathcal{F}(x,|u|) \diff x\Big)^{\frac{\mu_1}{\alpha_2}}, \Big(	\int_{ Q_{\eps}} \mathcal{F}(x,|u|) \diff x\Big)^{\frac{\mu_2}{\alpha_2}}\right\}.
\end{cases}
\end{equation*}
Therefore, by choosing
\begin{equation} \label{k*.2}
	k_\ast=( 4\bar{C})^{\frac{1}{\alpha_1}} b^{\frac{1}{\alpha_1} \left(\frac{1}{\mu_1}+\frac{\mu_2-\mu_1}{\mu_2}\right)} \max \left\{\Big(\int_{ Q_{\eps}}\mathcal{F}(x,|u|) \diff x\Big)^{\frac{\mu_1}{\alpha_2}}, \Big(	\int_{ Q_{\eps}} \mathcal{F}(x,|u|) \diff x\Big)^{\frac{\mu_2}{\alpha_1}}\right\},
\end{equation}
we obtain \eqref{Y0}; hence, \eqref{limYn.0} follows.

On the other hand, noting $D_n=A_{k_n} \cap Q_{\eps_n}$ we have
\begin{align*}
	Y_n  = \int_{D_n} \mathcal{F}(x,u-k_n) \diff x &  \longrightarrow \int_{A_{2k_\ast} \cap Q_{\frac{\eps}{3}}} \mathcal{F}(x,u-2k_\ast) \diff x=\int_{Q_{\frac{\eps}{3}}} \mathcal{F}\left(x,(u-2k_\ast)_+ \right) \diff x 
\end{align*}
as \ $n \to \infty$. Using this and \eqref{limYn.0} we get
\begin{equation*}
\int_{Q_{\frac{\eps}{3}}} \mathcal{F}\left(x,(u-2k_\ast)_+ \right) \diff x = 0.
\end{equation*}
From this we get  
\begin{equation*}
	u \leq 2k_\ast \quad \text{a.e. on} \ Q_{\frac{\eps}{3}}. 
\end{equation*}
By replacing $u$ with $-u$ in the arguments above, we also obtain
\begin{equation*}
	(-u) \leq 2k_\ast \quad \text{a.e. on} \ Q_{\frac{\eps}{3}}. 
\end{equation*}
From the last two inequalities and \eqref{k*.2} imply
\begin{equation*}
	\|u\|_{L^\infty(Q_{\frac{\eps}{3}})} \leq C \max \left\{\Big(\int_{ Q_{\eps}}\mathcal{F}(x,|u|) \diff x\Big)^{\frac{\mu_1}{\alpha_2}}, \Big(	\int_{ Q_{\eps}} \mathcal{F}(x,|u|) \diff x\Big)^{\frac{\mu_2}{\alpha_1}}\right\},
\end{equation*}
where $C$ is independent of $u$ and $x_0$. In other words, we arrive at \eqref{PTS.priori-l}.

From \eqref{PTS.priori-l} we easily obtain \eqref{Sub-bound} and the decay estimate \eqref{decay}. The proof is complete.

\end{proof}

\begin{proof}[Proof of Theorem~\ref{T.Sub2}] The proof is almost identical to that of Theorem~\ref{T.Sub1}. The only difference is that we apply Proposition~\ref{P.E1} instead of Proposition~\ref{P.E2}.
\end{proof}

\section{The boundedness of solution for critical problems}\label{Cri}

In this section, we investigate the boundedness of weak solutions to problem \eqref{eq1} when the reaction term possesses critical growths. Throughout this section, we assume that the assumption \eqref{O}  holds. Furthermore, we make the following assumption.

\begin{enumerate}[label=\textnormal{(H$_2$)},ref=\textnormal{H$_2$}]
	\item\label{H2}
	$\mathcal{A}\colon\RN \times\R\times\R^N\to \R^N$ and $\mathcal{B}\colon\RN \times \R\times \R^N\to \R$ are Carath\'eodory functions such that
	\begin{enumerate}[label=\textnormal{(\roman*)},ref=\textnormal{(H$_2$)(\roman*)}]
		\item
		$ |\mathcal{A}(x,t,\xi)|
		\leq M_1 \left[k(x)+ |t|^{p^*(x)-1}+a(x)^{\frac{q^*(x)}{q(x)}}|t|^{q^*(x)-1}+|\xi|^{p(x)-1} +a(x)|\xi|^{q(x)-1}\right],$
		\item\label{H2ii}
		$ \mathcal{A}(x,t,\xi)\cdot \xi \geq M_2 \left[|\xi|^{p(x)} +a(x)|\xi|^{q(x)}\right]-M_3\left[|t|^{p^*(x)}+a(x)^{\frac{q^*(x)}{q(x)}}|t|^{q^*(x)}+1\right],$
		\item\label{H2iii}
		$ |\mathcal{B}(x,t,\xi)|
		\leq L \left[ d(x)+ |t|^{p^*(x)-1}+a(x)^{\frac{q^*(x)}{q(x)}}|t|^{q^*(x)-1}+|\xi|^{\frac{p(x)}{(p^*)'(x)}} +a(x)^{\frac{N+1}{N}}|\xi|^{\frac{q(x)}{(q^*)'(x)}}\right],$ 
	\end{enumerate}
	for a.\,a.\,$x\in\RN$ and for all $(t,\xi) \in \R\times\R^N$, where $M_1,M_2,M_3$ and $L$ are positive constants; $k \in L_{\loc}^{\rho'(\cdot)}(\RN)$, $d \in L_{\loc}^{\sigma'(\cdot)}(\RN)$ with $\rho,\sigma \in C_+(\RN)\cap C^{0,1}(\RN)$ such that $\kappa:=\max\{\rho, \sigma \} \ll \frac{p^\ast}{q}$ and
	\begin{equation}\label{k-d.c}
		M:=\sup_{y \in \RN} \int_{Q_1(y)} \left(|k(x)|^{\rho'(x)}+|d(x)|^{\sigma'(x)}\right) \diff x<\infty.
	\end{equation}
\end{enumerate}

\medskip

A weak solution of problem \eqref{eq1} under the assumptions \eqref{O} and \eqref{H2} is also defined as in  Definition \ref{D.wsol1}. It is clear that all the integrals in \eqref{def.weak.sol} are well defined in view of Propositions \ref{P.E1} and \ref{P.E2}.

The boundedness for critical problems, which is our second main result, is stated as follows.
\begin{theorem} \label{T.Cri1}
	Let hypotheses \eqref{O} and \eqref{H2} be fulfilled. Then, any weak solution $u\in X_V$ of problem \eqref{eq1} belongs to $L^\infty(\RN)$. 
\end{theorem}
Again, we state the boundedness for a possibly weaker solution when $\essinf_{x \in \R^N}V(x)>0$. 
\begin{theorem}\label{T.Cri2}
	Let \eqref{D} and \eqref{H2} hold, and let $ V\in L^1_{\loc}(\RN)$ satisfy $\essinf_{x \in \R^N}V(x)>0$. Then, any weak solution $u\in W^{1,\mathcal{H}}_V(\RN)$ of problem \eqref{eq1} belongs to $L^\infty(\RN)$.
\end{theorem}
With the same reason as in the previous section, we only present a proof of Theorem~\ref{T.Cri1}.
\begin{proof}[Proof of Theorem~\ref{T.Cri1}]
	Let $u\in X_V$ be a weak solution of problem \eqref{eq1}, and $k_*>1$ be sufficiently large such that
	\begin{align}\label{k*}
		\int_{A_{k_{*}}}\mathcal{H}(x,|\nabla u|) \diff x+\int_{A_{k_{*}}}V(x)\mathcal{H}(x,u) \diff x+\int_{A_{k_{*}}}\mathcal{G}^*(x,u) \diff x <1,
	\end{align}
where $A_k$ ($k\in\R$) and $\mathcal{G}^*$ are defined as in \eqref{def.Ak} and \eqref{def.G}, respectively. 

Let $\{k_n\}_{n \in \N_0}$ be the sequence defined by  \eqref{def.kn}.  We  define the recursion sequence $\{Z_n\}_{n \in \N_0}$ as
\begin{align*}
	Z_n:=\int_{A_{k_n}}\mathcal{G}^*(x,u-k_n) \diff x ~ \text{for}~ n \in \N_0.
\end{align*}
Similar arguments to those leading to \eqref{def.Fn} give
\begin{align} \label{G*-Zn}
	\int_{A_{k_{n+1}}} \mathcal{G}^*(x,u) \diff x \leq  2^{(n+2)(q^\ast)^+}Z_n,\quad \forall n\in \N_0.
\end{align}

	Dividing $\RN$ into cubes $Q_i$ ($\in \N$), which  are disjoint with edges parallel to the coordinate axes and length of $\eps \in \left(0,\frac{1}{3}\right)$. Let $\wt{Q}_i$ be the cube that is concentric with $Q_i$ and has an edge length of $3\eps$. Note that each $\wt{Q}_i$ contains $3^N$ cubes $Q_j\ (j \in \N)$, and each $Q_i$ is contained in $3^N$ cubes $\wt{Q}_j\ (j \in \N)$.
For a function $h\in C(\RN)$ and each $i\in\N$, we denote
\begin{align*}
	h_i^+= \sup_{x\in Q_i} h(x), \quad 	h_i^-= \inf_{x\in Q_i} h(x),
\end{align*}
and 
\begin{align*}
	\wt{h}_i^+= \sup_{x\in \wt{Q}_i} h(x), \quad 	\wt{h}_i^-= \inf_{x\in \wt{Q}_i} h(x).
\end{align*}
By the uniform continuity of variable exponents, we may fix a sufficiently small $\eps$ such that
\begin{align}\label{C.loc.exp}
	q_i^+ + \delta  < (p^\ast)_i^-\ \ \text{and}\ \ \kappa_i^+ + \delta < \frac{(p^\ast)_i^-}{q^+_i}
	\quad \text{for all } i\in\N,
\end{align}
where
\begin{equation}
	\delta:=\frac{1}{2} \min \left\{\left(p^\ast-q\right)^-, \left(\frac{p^\ast}{q}- \kappa\right)^- \right\}>0.
\end{equation}
Let $n,i\in\N$ be arbitrary but fixed, and let $v_n:=(u-k_{n+1})_+$. In the rest of the proof, $C$ and $C_j$ ($j\in\N$) stand for positive constants independent of $u$, $k_\ast$, $Q_i$ and $n$. 

\medskip

\textbf{Claim 1.} It holds that
\begin{align}\label{G^*.E}
	\int_{Q_i }\mathcal{G}^*(x,v_n) \diff x\leq  C_1 \left[\int_{Q_i} \Big(\mathcal{H}(x,|\nabla v_n|) + \mathcal{H}(x,v_n)\Big) \diff x\right]^{\frac{(p^\ast)^-_i}{q^+_i}}.
\end{align}
Indeed, we first note that by \eqref{k*}, 
\begin{equation}\label{int<1}
\int_{\RN}\mathcal{G}^*(x,v_n)  \diff x=\int_{A_{k_{n+1}}}\mathcal{G}^*(x,v_n)  \diff x \leq \int_{A_{k_{*}}}\mathcal{G}^*(x,u) \diff x<1.
\end{equation}
Hence, we apply Propositions \ref{prop.nor-mod.D} and \ref{RmkEm} to obtain
\begin{align*}
	\int_{Q_i}\mathcal{G}^*(x,v_n) \, \diff x \leq \|v_n\|_{L^{\mathcal{G}^*}(Q_i)}^{(p^*)_i^-}\leq C_2 \|v_n\|_{W^{1,\mathcal{H}}(Q_i)}^{(p^*)_i^-}.
\end{align*}
On account of \eqref{m-n} and \eqref{k*} we easily derive \eqref{G^*.E} from the last estimate. Thus, Claim 1 is proved.

\medskip
\textbf{Claim 2.} 
It holds that\begin{align*}
	\int_{Q_i} &\Big[ \mathcal{H}(x,|\nabla v_n|) + \left(1+V(x)\right) \mathcal{H}(x,v_n) \Big] \diff x\\
	& \leq C_3  \left[\int_{\Lambda_{n,i}} \mathcal{G}^\ast (x,u) \diff x + \int_{\Lambda_{n,i}} |k(x)|u\diff x + \int_{\Lambda_{n,i}} |d(x)|u\diff x\right],
\end{align*}
where  
\begin{equation*}\label{Lambda.n.i}
	\Lambda_{n,i}:=A_{k_{n+1}} \bigcap \wt{Q}_i.
\end{equation*} 
Indeed, the claim can be proved in the same manner as in the proof leading to \eqref{local.est.1}. To this end, let $\xi \in C_c^\infty(\RN)$ be such that 
\begin{align}\label{def.xi}
	0 \leq \xi \leq 1, ~ \xi=1 ~ \text{on}~ Q_i, ~ \supp (\xi) \subset \wt{Q}_i ~ \text{and}~ |\nabla \xi|<\frac{C}{\eps}.
\end{align}
Then, by taking $v_n \xi^{\wt{q}^+_i}$ as a test function in \eqref{def.weak.sol} leads to
\begin{align*}
	&\int_{\Lambda_{n,i}}
	\left(\mathcal{A}(x,u,\nabla u) \cdot \nabla v_n\right)\xi^{\wt{q}^+_i}  \diff x +  \int_{\Lambda_{n,i}} V(x) \Big(|u|^{p(x)-2}+a(x)|u|^{q(x)-2}\Big)uv_n\xi^{\wt{q}^+_i} \diff x  \notag  \\
	& \quad \quad =\int_{\Lambda_{n,i}} \mathcal{B}(x,u,\nabla u)v_n\xi^{\wt{q}^+_i} \diff x 
	- \int_{\Lambda_{n,i}}	\left( \mathcal{A}(x,u,\nabla u) \cdot \nabla \xi \right) v_n  \wt{q}^+_i \xi^{\wt{q}^+_i-1} \, \diff x.
\end{align*}
Due to the fact that $\nabla u= \nabla v_n$ and $v_n < u $ a.e. on $\Lambda_{n,i}$ the preceding equality yields
\begin{align}\label{C.var.Eq}
	&\int_{\Lambda_{n,i}}
	\Big(\mathcal{A}(x,u,\nabla v_n) \cdot \nabla v_n\Big)\xi^{\wt{q}^+_i}   \diff x +  \int_{\Lambda_{n,i}} V(x) \mathcal{H}(x,v_n)\xi^{\wt{q}^+_i} \diff x  \notag \\
	& \quad \quad
	\leq \int_{\Lambda_{n,i}} \left| \mathcal{B}(x,u,\nabla v_n)v_n\xi^{\wt{q}^+_i}\right|  \diff x +   \wt{q}^+_i  \int_{\Lambda_{n,i}} \left| \left( \mathcal{A}(x,u,\nabla v_n) \cdot \nabla \xi \right)  v_n  \xi^{\wt{q}^+_i-1} \right| \diff x.
\end{align}
Invoking \eqref{H2} (ii) and arguing as that leading to \eqref{A.1} with noticing $u>k_{n+1}>1$ on $A_{k_{n+1}}$,  we deduce
\begin{align}\label{Est.A}
	\int_{\Lambda_{n,i}} \left(\mathcal{A}(x,u,\nabla v_n) \cdot \nabla v_n\right) \xi^{\wt{q}^+_i} \diff x &\geq M_2\int_{\Lambda_{n,i}}\mathcal{H}(x,|\nabla v_n|)\xi^{\wt{q}^+_i}  \diff x-M_3\int_{\Lambda_{n,i}}\left[\mathcal{G}^*(x,u)+1\right] \diff x \notag\\
	&\geq M_2\int_{\Lambda_{n,i}} \mathcal{H}(x,|\nabla v_n|)\xi^{\wt{q}^+_i} \diff x-2M_3\int_{\Lambda_{n,i}} \mathcal{G}^*(x,u) \diff x. 
\end{align}
Invoking \eqref{C.var.Eq} and \eqref{Est.A} we arrive at
\begin{multline}\label{Est.A.1}
\int_{\Lambda_{n,i}} \Big[\mathcal{H}(x,|\nabla v_n|) + \left(1+V(x)\right) \mathcal{H}(x,v_n)\Big] \xi^{\wt{q}^+_i} \diff x  \leq 	 C_4\int_{\Lambda_{n,i}} \left| \mathcal{B}(x,u,\nabla v_n)v_n\xi^{\wt{q}^+_i}\right|  \diff x\\ +   C_4  \int_{\Lambda_{n,i}} \left| \left( \mathcal{A}(x,u,\nabla v_n) \cdot \nabla \xi \right)  v_n  \xi^{\wt{q}^+_i-1} \right| \diff x  +C_4\int_{\Lambda_{n,i}} \mathcal{H}(x,v_n) \diff x +C_4\int_{\Lambda_{n,i}} \mathcal{G}^*(x,u) \diff x.
\end{multline}
By \eqref{H2} (iii), Young's inequality, \eqref{def.xi} and the fact that $v_n< u$ a.e. on $\Lambda_{n,i}$, we obtain
\begin{align}\label{Est.B}
	&C_4\int_{\Lambda_{n,i}} \left| \mathcal{B}(x,u,\nabla v_n)v_n\xi^{\wt{q}^+_i}\right|  \diff x \notag \\
	&\leq C_4L\int_{\Lambda_{n,i}}\left(|d(x)| + u^{p^*(x)-1} + a(x)^{\frac{q^*(x)}{q(x)}}u^{q^*(x)-1} + |\nabla v_n|^{\frac{p(x)}{(p^*)'(x)}} + a(x)^{\frac{N+1}{N}}|\nabla v_n|^{\frac{q(x)}{(q^*)'(x)}}\right)v_n\xi^{\wt{q}^+_i}  \diff x \notag\\
	&\leq \frac{1}{4}\int_{\Lambda_{n,i}}\mathcal{H}(x,|\nabla v_n|)\xi^{\wt{q}^+_i}  \diff x + C_5\int_{\Lambda_{n,i}}\mathcal{G}^*(x,u) \diff x + C_5 \int_{\Lambda_{n,i}} |d(x)|u  \diff x.
\end{align}
On the other hand, by \eqref{H2} (i) and the fact that $v_n< u$ a.e. on $\Lambda_{n,i}$ again, we have
\begin{multline}\label{A.nab.e}
C_4 \int_{\Lambda_{n,i}} \Big|\left( \mathcal{A}(x,u,\nabla v_n) \cdot \nabla \xi\right) v_n  \xi^{\wt{q}_i^+-1}\Big| \, \diff x  \leq  C_4M_1\int_{\Lambda_{n,i}} \left[u^{p^*(x)}+a(x)^{\frac{q^*(x)}{q(x)}}u^{q^*(x)} \right]|\nabla \xi|\xi^{\wt{q}^+_i-1}\, \diff x \\
 + C_4M_1 \int_{\Lambda_{n,i}} \left[|\nabla v_n|^{p(x)-1} +a(x)|\nabla v_n|^{q(x)-1}\right]|\nabla \xi|v_n\xi^{\wt{q}^+_i-1} \, \diff x  + 
C_4M_1\int_{\Lambda_{n,i}} |k(x)|v_n\diff x.
\end{multline}
By making use of \eqref{def.xi}, Young's inequality and $u>k_{n+1}>1$ on $\Lambda_{n,i}$ we argue in the same manner as those obtaining \eqref{S.est.2} and \eqref{S.est.1} to conclude
\begin{align}\label{Est.A.2}
	M_1C_4 \int_{\Lambda_{n,i}} \left[u^{p^*(x)}+a(x)^{\frac{q^*(x)}{q(x)}}u^{q^*(x)} \right]|\nabla \xi|\xi^{\wt{q}^+_i-1}\, \diff x &  \leq 
	C_6 \int_{\Lambda_{n,i}} \mathcal{G}^\ast(x,u) \diff x
\end{align}
and
\begin{align}\label{Est.A.3}
	&M_1C_4\int_{\Lambda_{n,i}} \left[|\nabla v_n|^{p(x)-1} +a(x)|\nabla v_n|^{q(x)-1}\right]|\nabla \xi|v_n\xi^{\wt{q}^+_i-1}\, \diff x  \notag \\
	& \quad \quad  \leq \frac{1}{4} \int_{\Lambda_{n,i}} \mathcal{H}(x,|\nabla v_n|)\xi^{\wt{q}^+_i}\diff x
	+ C_7 \int_{\Lambda_{n,i}} \left[u^{p(x)}+a(x)u^{q(x)}\right] \diff x \notag \\
	& \quad \quad  \leq \frac{1}{4} \int_{\Lambda_{n,i}} \mathcal{H}(x,|\nabla v_n|)\xi^{\wt{q}^+_i}\diff x
	+ C_7 \int_{\Lambda_{n,i}} \left[u^{p*(x)}+\left(a(x)u^{q(x)}\right)^{\frac{q^*(x)}{q(x)}}+1\right] \diff x \notag \\
	&  \quad \quad \leq \frac{1}{4} \int_{\Lambda_{n,i}} \mathcal{H}(x,|\nabla v_n|)\xi^{\wt{q}^+_i}\diff x + 	C_8 \int_{\Lambda_{n,i}}\mathcal{G}^\ast(x,u) \diff x.
\end{align}
On account of \eqref{A.nab.e}-\eqref{Est.A.3} we obtain
\begin{align} \label{Est.A.4}
\notag	C_4 \int_{\Lambda_{n,i}} &\Big| \left(\mathcal{A}(x,u,\nabla v_n) \cdot \nabla \xi\right) v_n \wt{q}^+_i \xi^{\wt{q}^+_i-1}\Big| \, \diff x\\
	  &\leq  \frac{1}{4} \int_{\Lambda_{n,i}} \mathcal{H}(x,|\nabla v_n|)\xi^{\wt{q}^+_i}\diff x  + C_9\left[ \int_{\Lambda_{n,i}} \mathcal{G}^\ast(x,u) \diff x  + \int_{\Lambda_{n,i}} |k(x)|u\diff x\right].
\end{align}
Using the fact that $\max\{1,v_n\}< u$ a.e. on $\Lambda_{n,i}$ again, we have
\begin{equation*}
	C_4\int_{\Lambda_{n,i}} \mathcal{H}(x,v_n) \diff x\leq C_4 \int_{\Lambda_{n,i}} \left[v_n^{p(x)}+a(x)^{\frac{q^*(x)}{q(x)}}v_n^{q^*(x)}+1 \right]\, \diff x \leq 2C_4 \int_{\Lambda_{n,i}} \mathcal{G}^\ast (x,u) \diff x.
\end{equation*}
Utilizing this, \eqref{Est.B} and \eqref{Est.A.4}, we derive Claim 2 from \eqref{Est.A.1}.

\medskip
We are now in a position to establish a recursion inequality for $\{Z_n\}_{n \in \N_0}$. Using Claims 1 and 2, we get
\begin{align}\label{G*vn}
	\int_{Q_i }\mathcal{G}^*(x,v_n) \diff x\leq C_{10}\left[\int_{\Lambda_{n,i}} \mathcal{G}^\ast (x,u) \diff x + \int_{\Lambda_{n,i}} |k(x)|u\diff x + \int_{\Lambda_{n,i}} |d(x)|u\diff x \right]^{\frac{(p^*)_i^-}{q_i^+}}.
\end{align}
Here we can take $C_{10}=C_1\left(C_3^{\frac{(p^*)^+}{q^-}}+1\right)$. We now estimate each term on the right-hand side of \eqref{G*vn}. Arguing as that obtaining \eqref{es.d(x)} and using \eqref{k-d.c}, \eqref{C.loc.exp} and the fact that $u>k_*>1$ on $\Lambda_{n,i}$, we infer
\begin{align}\label{cri.k(x)}
	\int_{\Lambda_{n,i}} |k(x)|u  \diff x & \leq  C_{11}\|k\|_{L^{\rho'(\cdot)}(\Lambda_{n,i})} \left(\int_{\Lambda_{n,i}} u^{\kappa^+_i} \diff x\right)^{\frac{1}{\kappa^+_i}}  \leq C_{12}\left[\int_{\Lambda_{n,i}} \mathcal{G}^\ast(x,u) \diff x \right]^{\frac{1}{\kappa_i^+}}. 
\end{align}
Similarly,
\begin{align}\label{cri.d(x)}
	\int_{\Lambda_{n,i}} |d(x)|u  \diff x \leq C_{13} \left[\int_{\Lambda_{n,i}} \mathcal{G}^\ast(x,u) \diff x \right]^{\frac{1}{\kappa_i^+}}.
\end{align}
Using \eqref{cri.k(x)}, \eqref{cri.d(x)},  \eqref{int<1} and \eqref{C.loc.exp},  it follows from \eqref{G*vn} that
\begin{equation*}
	\int_{Q_i} \mathcal{G}^\ast(x,v_n)  \diff x  \leq C_{14}  \left[\int_{\Lambda_{n,i}} \mathcal{G}^\ast(x,u) \diff x \right]^{\frac{(p^*)_i^-}{\kappa_i^+q_i^+}} \leq C_{14}   \left[\int_{\Lambda_{n,i}} \mathcal{G}^\ast(x,u) \diff x \right]^{1+\delta_0},
\end{equation*}
where $C_{14}=1+C_{12}+C_{13}$ and $\delta_0:=\frac{\delta}{\kappa^+}$. Moreover, from \eqref{G*-Zn} it holds that
\begin{equation*}
	\sum_{i \in \N}\int_{\Lambda_{n,i}} \mathcal{G}^\ast(x,u) \diff x \leq 3^N 	\sum_{i \in \N} \int_{A_{k_{n+1}} \cap Q_i}   \mathcal{G}^\ast(x,u) \diff x = 3^N\int_{A_{k_{n+1}}} \mathcal{G}^\ast(x,u) \diff x \leq 3^N2^{2(q^\ast)^+} 2^{n(q^\ast)^+}Z_n.
\end{equation*}
Combining the last two inequalities gives
\begin{align*}
	Z_{n+1} & = \sum_{i \in \N} \int_{Q_i} \mathcal{G}^\ast(x,v_n) \diff x \leq C_{14} \sum_{i \in \N} \left[\int_{\Lambda_{n,i}} \mathcal{G}^\ast(x,u) \diff x \right]^{1+\delta_0}  \\
	& \leq C_{14} \left[\sum_{i \in \N} \int_{\Lambda_{n,i}} \mathcal{G}^\ast(x,u) \diff x  \right] ^{1+\delta_0} \leq C_{15} 2^{n(q^\ast)^+ (1+\delta_0)}Z_n^{1+\delta_0}.
\end{align*}
Thus, we have the recursion inequality
\begin{align} \label{Recur}
	Z_{n+1}& \leq C_{15}b^n Z_n^{1+\delta_0}, \quad \forall n\in\N_0,
\end{align}
where $b:=2^{(q^\ast)^+(1+\delta_0)}$.
By means of  Lemma \ref{leRecur}, it holds that 
\begin{align}\label{Recur+1}
	Z_n \to 0 \quad \text {as } n\to \infty
\end{align}
provided 
\begin{align}\label{Z_0}
	Z_0\leq \min\left\{1,C_{15}^{-\frac{1}{\delta_0}}\ b^{-\frac{1}{\delta_0^{2}}}\right\}.
\end{align}
Note that
\begin{align*}
	Z_0 = \int_{A_{k_{*}}}\mathcal{G}^*(x,u-k_\ast) \diff x \leq \int_{A_{k_{*}}}\mathcal{G}^*(x,u) \diff x.
\end{align*}
Thus, by choosing $k^\ast>1$ sufficiently large such that \eqref{k*} holds and
\begin{align*}
	\int_{A_{k_{*}}}\mathcal{G}^*(x,u) \diff x
	\leq \min\left\{1,C_{15}^{-\frac{1}{\delta_0}}\ b^{-\frac{1}{\delta_0^{2}}}\right\},
\end{align*}
we obtain \eqref{Z_0}, hence, \eqref{Recur+1} is fulfilled. Consequently, we have
\begin{align*}
	\int_{\RN}\mathcal{G}^*(x,(u-2k_\ast)_+) \diff x=\lim_{n\to\infty}Z_n=0.
\end{align*}
Thus, $(u-2k_*)_{+}=0$ a.e.\,in $\RN$, and so
\begin{align*}
	\esssup_{\RN} u \leq 2k_*.
\end{align*}
Replacing $u$ by $-u$ in the above arguments we also get that
\begin{align*}
	\esssup_{\RN} (-u) \leq 2k_*.
\end{align*}
From the last two estimates we obtain
\begin{align*}
	\|u\|_{L^{\infty}(\RN)}\leq 2k_{*}.
\end{align*}
The proof is complete.

\end{proof}
We conclude this section with the remark that, if the reaction term exhibits growth not exceeding $p^*$, as in existing works such as \cite{RK25}, then the regularity assumptions on $p$, $q$, and $a$ can be weakened. More precisely, Theorems~\ref{T.Cri1} and \ref{T.Cri2} remain valid if assumption \eqref{H2}(iii) is replaced by
\begin{equation}\label{H2(iii')}
	|\mathcal{B}(x,t,\xi)|
	\leq L \left[ d(x)+ |t|^{p^*(x)-1}+|\xi|^{\frac{p(x)}{(p^*)'(x)}} +a(x)^{\frac{1}{q(x)}+\frac{1}{(p^*)'(x)}}|\xi|^{\frac{q(x)}{(p^*)'(x)}}\right]
\end{equation}
and assumption \eqref{D} is replaced by
\begin{enumerate}[label=\textnormal{(D1)},ref=\textnormal{D1}]
	\item\label{D1}
	$p,q\in UC(\RN) $ such that $1<p(x)<N$, $p(x)\leq q(x)\leq p^*(x)$  for all $x\in \RN,$ $p\in  C^{0, \frac{1}{|\log t|}}(\RN)$  and $0  \leq a(\cdot) \in L^\infty(\RN)$.
\end{enumerate}
Here, $ UC(\RN)$ denotes the set of all uniformly continuous functions on $\RN$, and $C^{0, \frac{1}{|\log t|}}(\RN)$ denotes the set of all functions $h\colon \RN \to \R$ that are log-H\"older continuous, that is, there exists $C>0$ such that
\begin{align*}
	|h(x)-h(y)| \leq \frac{C}{|\log |x-y||}\quad\text{for all } x,y\in \RN \text{ with } |x-y|<\frac{1}{2}.
\end{align*}
Similarly, Theorems~\ref{T.Sub1} and \ref{T.Sub2} remain valid if \eqref{H1}(iii) is replaced by \eqref{H2(iii')} with $r\in UC(\RN)$ in place of $p^*$ and \eqref{D} is replaced by \eqref{D1}. To establish these facts, in the proofs of Theorems~\ref{T.Cri1} and \ref{T.Cri2} (resp. Theorems~\ref{T.Sub1} and \ref{T.Sub2}) we consider  $\mathcal{G}^*(x,t)=t^{p^*(x)}$ (resp. $\mathcal{F}(x,t)=t^{r(x)}$) and apply the embeddings $W^{1,\mathcal{H}}(Q_i)\hookrightarrow W^{1,p(\cdot)}(Q_i)\hookrightarrow L^{p^*(\cdot)}(Q_i)\hookrightarrow L^{r(\cdot)}(Q_i)$ (see \cite[Corollary 8.3.2]{Diening}) instead of Proposition~\ref{RmkEm}. 

For the case of growth exceeding $p^*$, the proofs of our results are essentially based on the embedding in \eqref{Embed.W_V0}, which requires the Lipschitz continuity of $p$, $q$ and $a$. However, it seems that these regularity assumptions may be weakened, as discussed above.


\section{Existence for supercritical double phase problems}\label{A.Sup}

In this section, we investigate the existence of weak solutions to double phase problems involving supercritical growth, as an application of the \textit{a priori} bounds established in Section~\ref{Sub} and  \cite{HW2022}. 
\subsection{Supercritical double phase problems on $\RN$} We consider the following problem
\begin{align}\label{eq.super}
	- \operatorname{div} \mathcal{A}_0(x,\nabla u) + V(x)A(x,u)
	= f_\theta(x,u)  \text{ in } \mathbb{R}^N,
\end{align}
where
\begin{equation}\label{A0}
	\mathcal{A}_0(x,\xi):=|\xi|^{p(x)-2} \xi + a(x) |\xi|^{q(x)-2}\xi,\ \ A(x,t)=|t|^{p(x)-2}t + a(x) |t|^{q(x)-2}t,
\end{equation}
and 
\begin{equation}\label{B0}
f_\theta(x,t):=b(x)|t|^{\gamma(x)-2}t + \theta  c(x)|t|^{r(x)-2}t 
\end{equation}
for $x,\xi \in \RN$, $t \in \R$ while $\theta>0$ is a real parameter, under the assumption \eqref{O} for the main operator and  the following assumption for the reaction term:
\begin{enumerate}[label=\textnormal{(F$_1$)},ref=\textnormal{F$_1$}]
	\item\label{f1} $0<b(\cdot), c(\cdot)\in L^\infty(\RN) \cap L^{\frac{p^\ast(\cdot)}{p^\ast(\cdot)-\gamma(\cdot)}}(\RN)$, with $\gamma \in C_+(\RN)$ satisfying $\gamma \ll p^\ast$ and  $q^+<\gamma^-\leq \gamma(x)\leq r(x)$ for all $ x \in \RN$.
\end{enumerate}
It is worth noting that the second term in the nonlinearity $f_\theta$ has a growth that can exceed the threshold $q^\ast$, i.e.,  it may have so-called \textit{supercritical} growth. By a weak solution of \eqref{eq.super} we mean a function $u \in X_V$ satisfying $f_\theta(\cdot,u)\in L_{\loc}^\infty(\R^N)$ and 
\begin{equation*}
	\int_{\RN} \mathcal{A}_0(x,\nabla u) \cdot \nabla v \diff x + \int_{\RN} V(x)A(x,u)v \diff x = \int_{\RN} f_\theta(x,u) v \diff x,\quad \forall v \in C_c^\infty(\RN).
\end{equation*}
Our main result in this section is stated as follows.


\begin{theorem}\label{T.Super.E}
	Let \eqref{O} and \eqref{f1} hold. Then, there exists  $\theta_0>0$, such that for every $\theta \in (0,\theta_0)$, problem \eqref{eq.super} admits a nontrivial nonnegative weak solution $u \in X_V \cap L^\infty(\RN)$.
\end{theorem} 

In the rest of this subsection, we always assume that all hypotheses of Theorem \ref{T.Super.E} hold. We shall prove Theorem \ref{T.Super.E} via a truncation technique. For this purpose, let $t_*\geq 1$ be fixed (to be specified later). We  consider a modified nonlinearity
$g:\RN \times \R \to \R$ defined by
\begin{equation*}
	g(x,t):=	\begin{cases}
		0 \ & \text{ if }  \ t < 0,\\
		b(x)t^{\gamma(x)-1} +  \theta c(x)t^{r(x)-1} \  & \text{ if }  \  0 \leq  t \leq t_*, \\
		b(x)t^{\gamma(x)-1} +  \theta c(x) t_\ast^{r(x)-\gamma(x)}  t^{\gamma(x)-1}  \  &\text{ if }  \   t_*<t
	\end{cases} 
\end{equation*}
for a.a. $x\in\R^N$. Then, $g$ is a  Carath\'eodory function on $\RN \times \R$ satisfying
\begin{enumerate}[label=\textnormal{(G)},ref=G]
	\item \label{G} 
	\begin{enumerate}[(i)]
		\item  $|g(x,t)|  \leq m(x) (1 + \theta C_{t_*})  |t|^{\gamma(x)-1}$ for  a.a. $x \in \RN$ and all $t \in \R$, where $$m(x):=b(x)+c(x) \text{ and } C_{t_*}:=t_\ast^{(r-\gamma)^+};$$
		\item   $0 \leq  \gamma(x) G(x,t) \leq  g(x,t)t$ for  a.a. $x \in \RN$ and all $t \in \R$ with 
		\begin{equation}\label{G.def}
			\displaystyle G(x,t):= \int_0^t g(x,s) \diff s.
		\end{equation} 
	\end{enumerate}
\end{enumerate}
Thus, it follows that $0<m(\cdot)\in L^\infty(\RN) \cap L^{\frac{p^\ast(\cdot)}{p^\ast(\cdot)-\gamma(\cdot)}}(\RN)$ and
\begin{equation}\label{G.bound}
	0\leq G(x,t) \leq m(x)(1+\theta C_{t^\ast})|t|^{\gamma(x)}
\end{equation}
for a.a. $x \in \RN$ and all $t \in \R$.

We consider the following modified problem with subcritical growth:
\begin{equation} \label{eq.m.s} 
	- \operatorname{div} \mathcal{A}_0(x,\nabla u) + V(x)A(x,u) =  g(x,u)  \text{ in } \RN.
\end{equation}
Note that any weak solution of \eqref{eq.m.s} is nonnegative a.a. in $\mathbb{R}^N$. Indeed, if $u$ is a weak solution of \eqref{eq.m.s}, then by Lemma~\ref{lem.(u-k)+} we can take $u_-$ as a test function for \eqref{eq.m.s} to obtain
\begin{equation*}
	\int_{\RN} \left[\mathcal{A}_0(x,\nabla u) \cdot \nabla u_- + V(x)A(x,u)u_-\right] \diff x=\int_{\RN} \big[\mathcal{H}\left(x,|\nabla u_-|\right) + V(x)\mathcal{H}\left(x,u_-\right)\big] \diff x =0.
\end{equation*}
This yields $u_-=0$, i.e., $u\geq 0$ a.a. in $\RN$. Hence, any weak solution $u$ of problem \eqref{eq.m.s} satisfying $\|u\|_\infty<t_*$ is also a bounded nonnegative weak solution of the original problem \eqref{eq.super}.

\begin{lemma}\label{Le.SoMo}
	For all $\theta>0$,	problem \eqref{eq.m.s} admits a nontrivial nonnegative weak solution $u_*\in X_V$ with
		\begin{equation}\label{Le5.2}
		\|u_*\| \leq C_1,
	\end{equation}
	where $C_1>0$ is independent of $u_*$ and $\theta$.
\end{lemma}
\begin{proof}
	The existence of a nontrivial nonnegative weak solution can be established via the Mountain Pass theorem. Since the argument follows a standard procedure, we only provide a sketch of the proof. We begin by considering the energy functional associated with problem~\eqref{eq.m.s}:
	\begin{align*}
		J_1(u):=\int_{\RN} \wh{A}(x,|\nabla u|) \diff x + \int_{\RN} V(x)\wh{A}(x,u)\diff x- \int_{\RN} G(x,u) \diff x, \quad u\in X_V, 
	\end{align*}
	where $\wh{A}:\RN \times \R \to \R$ is defined by
	\begin{equation}\label{A.hat}
	\wh{A}(x,t):=\frac{|t|^{p(x)}}{p(x)} + a(x) \frac{|t|^{q(x)}}{q(x)}.
	\end{equation}
In view of Proposition \ref{P.E2}, we have that $J_1 \in C^1(X_V,\R)$ with derivative given by
		\begin{equation*}
		\langle J_1'(u),v \rangle = \int_{\RN} \mathcal{A}_0(x,\nabla u) \cdot \nabla v \diff x + \int_{\RN}V(x)A(x,u)v \diff x -\int_{\RN} g(x,u)v \diff x,
	\end{equation*}
	for all $u,v \in X_V$. In order to show \eqref{Le5.2}, we also consider the functional
	\begin{equation}\label{J2}
		J_2(u):= \int_{\RN} \wh{A}(x,|\nabla u|) \diff x + \int_{\RN} V(x)\wh{A}(x,u)\diff x - \int_{ \RN} \frac{b(x)}{\gamma(x)}u_+^{\gamma(x)} \diff x,\quad u\in X_V.
	\end{equation} 
	Since $q^+<\gamma^-$, we can find $\bar{\phi} \in C_c^\infty(\RN) $,   $\bar{\phi}(\cdot) \geq 0$ in $\RN$ such that $\|\bar{\phi}\| > 0$ and $J_2(\bar{\phi})<0$.  Note that we have
	\begin{equation}\label{Le5.2C}
	J_1(u)\leq J_2(u),\quad\forall u\in X_V.
	\end{equation}
	 Hence, we obtain
	\begin{equation}\label{MP1}
		J_1(\bar{\phi})<0.
	\end{equation}
		By $q^+<\gamma^-$ again, we can argue as in \cite[Lemma 4.4]{HH25} to show that 
		\begin{equation}\label{MP2}
		\text{there exist} \ \rho>0 \ \text{and} \ 0<\delta<\|\bar{\phi}\|\  \text{such that} \ J_1(u) > \rho\  \text{if} \ \|u\|=\delta. 
		\end{equation}
	For each $\theta>0$, define 
	\begin{equation}\label{c0.theta}
		c_{\theta}:= \inf_{\ell \in \Gamma}\max_{t \in [0,1]}J_1(\ell(t)),
	\end{equation}
	where 
	\begin{equation}\label{Ga}
	\Gamma:=\left\{ \ell \in C([0,1],X_V):~ \ell(0)=0,~ \ell(1)= \bar{\phi}\right\}.
	\end{equation}
	Then, by invoking \cite[Lemma 3.1]{Garcia}, we deduce from the Mountain Pass geometry \eqref{MP1}-\eqref{MP2} of $J_1$ that there exists a $\textup{(PS)}_{c_{\theta}}$ sequence $\{u_n\}_{n \in \N} \subset X_V$ for $J_1$. Thus,
	we have that for $n$ large,
	\begin{align*}
		\notag
		c_\theta+1+\|u_n\|& \geq J_1(u_n)-\frac{1}{\gamma^-}\scal{J_1'(u_n),u_n}\\ \notag
		&\geq \round{\frac{1}{q^+}-\frac{1}{\gamma^-}} \left[ \int_{\mathbb{R}^N}\mathcal{H}(x,|\nabla u_n|) \diff x + \int_{\mathbb{R}^N} V (x)\mathcal{H}(x,u_n) \diff x \right]    \\ 
		&\geq  \round{\frac{1}{q^+}-\frac{1}{\gamma^-}}\left(\|u_n\|^{p^-}-1\right).
	\end{align*}
	This implies that $\{u_n\}_{n \in \N}$ is bounded in $X_V$ since $p^->1$.
	Then, according to Proposition~\ref{P.E2}, it holds, up to a subsequence, that
	\begin{gather*}
		u_n \rightharpoonup u_*  \quad \text{in} \  X_V\ \  \text{and}\ \ u_n(x) \to u_*(x) \quad \text{a.a.} \ \ x \in\mathbb{R}^N.
	\end{gather*}  
	From this and \cite[Theorem 1.14]{HH25}, noting $m \in L^{\frac{p^\ast(\cdot)}{p^\ast(\cdot)-\gamma(\cdot)}}(\RN)$, we can argue as \cite[Proof of Theorem 4.2]{HH25} to obtain
	\begin{align*}
		\lim_{ n \to \infty} \left[\int_{\RN} \mathcal{A}_0(x,\nabla u_*)\cdot \nabla (u_n-u_*) \diff x + \int_{\RN} V(x)A(x,u)(u_n-u_*) \diff x \right]= 0.
	\end{align*} 
Thus, $u_n\to u_*$ in $X_V$ in view of \cite[Lemma 2.15]{HH25}. Consequently, $J_1'(u_*)=0$ and $J_1(u_*)=c_{\theta}\geq \delta>0$, and thus, $u_*$ is a nontrivial nonnegative weak solution to problem \eqref{eq.m.s}.

Finally, we show \eqref{Le5.2}. Set
\begin{equation*}
	c_0:=\inf_{\ell \in \Gamma} \max_{t \in [0,1]} J_2(\ell(t))
\end{equation*}
where $\Gamma$ is given by \eqref{Ga}. By \eqref{Le5.2C}, it holds that
\begin{equation*}
	c_{\theta} \leq c_0.
\end{equation*}
On the other hand, by \eqref{G} (ii) and \eqref{m-n} it follows that
\begin{align*}
	c_0 & \geq 	c_{\theta}  = J_1(u_*) - \frac{1}{\gamma^-} \langle J'_1(u_*),u_*  \rangle  \\
	& \geq \left(\frac{1}{q^+} - \frac{1}{\gamma^-}\right) \min\left\{\|u_*\|^{p^-},1\right\}.
\end{align*}
This yields
$$\|u_*\|\leq \max\left\{1, c_0^{\frac{1}{p^-}}\left(\frac{1}{q^+} - \frac{1}{\gamma^-}\right)^{-\frac{1}{p^-}}\right\}.$$
The proof is complete.
\end{proof}

Now, set $C_m:= 2\|m\|_{L^\infty(\RN)}$ and consider the following condition
\begin{enumerate}[label=\textnormal{(G$_1$)},ref=G$_1$]
	\item \label{G1} $|g(x,t)|  \leq  C_m|t|^{\gamma(x)-1}$ for all $t \in \R$ and a.a. $x \in \RN$.
\end{enumerate}
Then, by applying Theorem~\ref{T.Sub1} for $\mathcal{F}_\gamma(x,t):=\mathcal{F}(x,t)=t^{r(x)} +a(x)^{\frac{s(x)}{q(x)}}t^{s(x)}$ with $r=\gamma$ and $s=\frac{\gamma+p^*}{2}$, we obtain the following lemma.
\begin{lemma}\label{Le.b}
	Let \eqref{G1} hold and $u$ be any weak solution of problem \eqref{eq.m.s}. Then $u \in L^\infty(\RN)$ and satisfies the following a priori bound
	\begin{equation*}
		\|u\|_{L^\infty(\RN)} \leq C_0  \max \left\{ \|u\|^{\tau_1}_{L^{\mathcal{F}_\gamma}(\RN)}, \|u\|^{\tau_2}_{L^{\mathcal{F}_\gamma}(\RN)}\right\}
	\end{equation*} 
	for some positive constants $C_0, \tau_1,\tau_2$ independent of $u$ and $\theta$. 
\end{lemma}

We are now in a position to prove Theorem~\ref{T.Super.E}.

\begin{proof}[Proof of Theorem~\ref{T.Super.E}]
By Proposition \ref{P.E2}, we find $C_\gamma>0$ such that
\begin{equation}\label{PT5.1-1}
\|u\|_{L^{\mathcal{F}_\gamma}(\RN)}\leq C_\gamma\|u\|,\quad \forall u\in X_V.
\end{equation}
Now, we choose 
$$t_*=\max \left\{1,C_0 (C_\gamma C_1)^{\tau_1}, C_0(C_\gamma C_1)^{\tau_2}\right\}\ \ \text{and}\ \ \theta_0:=C_{t_*}^{-1},$$ where $C_0, \tau_1,\tau_2$ are given in Lemma~\ref{Le.b} while $C_1$ is given in Lemma~\ref{Le.SoMo}. 

Let $\theta\in (0,\theta_0)$. Then, problem~\eqref{eq.m.s} admits a nontrivial nonnegative weak solution $u_*\in X_V$ with $	\|u_*\| \leq C_1$ in view of Lemma~\ref{Le.SoMo}. On the other hand,  \eqref{G1} is fulfilled thanks to the fact that $\theta C_{t_\ast}<1$ and \eqref{G}. Thus, by Lemma \ref{Le.b} and \eqref{PT5.1-1}, we obtain
\begin{equation*}
	\|u_*\|_{L^\infty(\RN)} \leq C_0  \max \left\{ \|u_*\|^{\tau_1}_{L^{\mathcal{F}_\gamma}(\RN)}, \|u_*\|^{\tau_2}_{L^{\mathcal{F}_\gamma}(\RN)}\right\}\leq C_0\max \left\{ (C_\gamma C_1)^{\tau_1}, (C_\gamma C_1)^{\tau_2}\right\}.
\end{equation*} 
Namely, we arrive at
	\begin{equation*}
		\|u_*\|_{L^\infty(\RN)} \leq t_*.
	\end{equation*}
	That is to say, $u_*$ is a nontrivial nonnegative weak solution of problem \eqref{eq.super}. The proof is complete.
\end{proof}

\subsection{Supercritical double phase problems on  bounded domains} We consider the following problem
\begin{align}\label{eq.super2}
\begin{cases}
		- \operatorname{div} \mathcal{A}_0(x,\nabla u) 
	= f_\theta(x,u)& \ \text{ in } \Omega,\\
	u=0&\ \text{ on }\ \partial\Omega,
\end{cases}
\end{align}
where $\Omega$ is a bounded domain in $\RN$ with Lipschitz boundary $\partial\Omega$, $\mathcal{A}_0$ and $f_\theta$ are given by \eqref{A0} and \eqref{B0}, respectively.  Let \eqref{D} hold, and  let $f_\theta$ satisfy
\begin{enumerate}[label=\textnormal{(F$_2$)},ref=\textnormal{F$_2$}]
	\item\label{H2'} $0<b(\cdot), c(\cdot)\in L^\infty(\Omega)$, with $\gamma \in C_+(\overline{\Omega})$ satisfying $\gamma (x) < p^\ast(x)$ and  $q^+<\gamma^-\leq \gamma(x)\leq r(x)$ for all $ x \in \overline{\Omega}$.
\end{enumerate}
By a weak solution of \eqref{eq.super2} we mean a function $u \in W_0^{1,\mathcal{H}}(\Omega)$ satisfying $f_\theta(\cdot,u)\in L_{\loc}^\infty(\Omega)$ and 
\begin{equation*}
	\int_{\Omega} \mathcal{A}_0(x,\nabla u) \cdot \nabla v \diff x = \int_{\Omega} f_\theta(x,u) v \diff x,\quad \forall v \in C_c^\infty(\Omega).
\end{equation*}
We have the following.


\begin{theorem}\label{T.Super.E2}
	Let \eqref{D} and \eqref{H2'} hold. Then, there exists  $\theta_0>0$, such that for every $\theta \in (0,\theta_0)$, problem \eqref{eq.super2} admits a nontrivial nonnegative weak solution $u \in W_0^{1,\mathcal{H}}(\Omega) \cap L^\infty(\Omega)$.
\end{theorem} 
This result is proved in the same manner as Theorem~\ref{T.Super.E}. More precisely, we consider  $\Phi_1,\Phi_2: W_0^{1,\mathcal{H}}(\Omega)\to\R$ defined by 
\begin{align*}
	\Phi_1(u):=\int_{\Omega} \wh{A}(x,|\nabla u|) \diff x - \int_{\Omega} G(x,u) \diff x, \quad u\in W_0^{1,\mathcal{H}}(\Omega), 
\end{align*}
	\begin{equation*}
	\Phi_2(u):= \int_{\Omega} \wh{A}(x,|\nabla u|) \diff x - \int_{\Omega} \frac{b(x)}{\gamma(x)}u_+^{\gamma(x)} \diff x,\quad u\in W_0^{1,\mathcal{H}}(\Omega),
\end{equation*} 
where, $\wh{A}$ and $G$ are defined as in \eqref{A.hat} and \eqref{G.def}, respectively while $W_0^{1,\mathcal{H}}(\Omega)$ is endowed with the equivalent norm $\|\cdot\|=\|\nabla\cdot\|_{L^{\mathcal{H}}(\Omega)}$ (see \cite[Proposition 2.18]{crespo2022new}). Then, we arrive at the conclusion of Theorem~\ref{T.Super.E} by repeating the proof of Theorem~\ref{T.Super.E} above, replacing $X_V$, $J_1$ and $J_2$ with $W_0^{1,\mathcal{H}}(\Omega)$, $\Phi_1$ and $\Phi_2$, respectively, and applying Proposition~\ref{P.Eb} and \cite[Theorem 4.2]{HW2022} in place of \cite[Theorem 2.14]{HH25} and Theorem~\ref{T.Sub1}, respectively. We leave the details to the reader.

\vspace{2mm}

{\bf{Acknowledgment}}
Ky Ho was supported by the University of Economics Ho Chi Minh City (UEH), Vietnam, through a university-level research project (CTD-CS-2024-31). Inbo Sim was supported by the National Research
Foundation of Korea Grant funded by the Korea Government (MEST) (NRF-2021R1I1A3A0403627011).

\vspace{2mm}

{\bf{Data Availability}} Data sharing is not applicable to this paper as no datasets were analyzed or generated.

\vspace{2mm}

{\bf{\LARGE Declarations}}

\vspace{2mm}

{\bf{Conflict of interest}} The authors have no Conflict of interest to declare that are relevant to the content of this
article.



\begin{thebibliography}{99}







\bibitem{Ba-Ra-Re} A. Bahrouni, V.D. Radulescu,  D.D. Repov\u{s},
Double phase transonic flow problems with variable growth: nonlinear patterns and stationary waves, Nonlinearity 32 (7) (2019) 2481--2495.

\bibitem{BCM15}
P. Baroni, M. Colombo, G. Mingione,
Harnack inequalities for double phase functionals,
Nonlinear Anal. 121 (2015) 206--222.

\bibitem{BCM16}
P. Baroni, M. Colombo, G. Mingione,
Non-autonomous functionals, borderline cases and related function classes,
St. Petersburg Math. J. 27 (2016) 347--379.

\bibitem{BCM18}
P. Baroni, M. Colombo, G. Mingione,
Regularity for general functionals with double phase,
Calc. Var. Partial Differential Equations 57 (2) (2018), Art. 62, 48 pp.

\bibitem{BKM15}
P. Baroni, T. Kuusi, G. Mingione,
Borderline gradient continuity of minima,
J. Fixed Point Theory Appl. 15 (2) (2014) 537--575.


\bibitem{Bar.Wang2001} T. Bartsch, A. Pankov, Z.-Q. Wang, Nonlinear Schr\"odinger equations with steep potential well, Commun. Contemp. Math. 3 (2001) 549--569.

\bibitem{BW.95}  T. Bartsch, Z.-Q. Wang, Existence and multiplicity results for superlinear elliptic problems on $\RN$, Comm. Partial Differential Equations 20 (1995) 1725--1741.

\bibitem{Ben2000}
V. Benci, P. D'Avenia, D. Fortunato, L. Pisani, Solitons in several space dimensions: Derrick’s problem and infinitely many solutions, Arch. Ration. Mech. Anal. 154 (4) (2000) 297--324.

\bibitem{Brezis2011}
H. Brezis, Functional Analysis, Sobolev Spaces and Partial Diﬀerential Equations, in: Universitext, Springer, New York, 2011.


	\bibitem{BO20}
S.-S. Byun, J. Oh,
Regularity results for generalized double phase functionals,
Anal. PDE { 13} (5) (2020)  1269--1300.

\bibitem{CT24} S. Carl, H. Tehrani, Global  $L^\infty$-estimate for general quasilinear elliptic equations in arbitrary domains of  $\mathbb{R}^N$, Partial Differ. Equ. Appl. 5 (3) (2024), Paper No. 15, 15 pp.

\bibitem{CLR05} Y. Chen, S. Levine, M. Rao, Variable exponent, linear growth functionals in image restoration, SIAM J. Appl. Math. 66
(4) (2006) 1383–1406.


\bibitem{Che2005} L. Cherfils, Y. Il’yasov, On the stationary solutions of generalized reaction diffusion equations with $p\&q$-Laplacian,
Commun. Pure Appl. Anal. 4 (1) (2005) 9--22.

\bibitem{CS16}
F. Colasuonno, M. Squassina,
Eigenvalues for double phase variational integrals,
Ann. Mat. Pura Appl. 195 (6) (2016)  1917--1959.


\bibitem{CM15a}
M. Colombo,  G. Mingione,
Bounded minimisers of double phase variational integrals,
Arch. Ration. Mech. Anal. 218 (1) (2015) 219--273.

\bibitem{CM15b}
M. Colombo,  G. Mingione,
Regularity for double phase variational problems,
Arch. Ration. Mech. Anal. 215 (2) (2015)  443--496.

\bibitem{Cor-Fig06} F.J.S.A. Corr\^ea, G.M. Figueiredo, On an elliptic equation of $p$-Kirchhoff type via variational methods, Bull. Austral. Math. Soc. 74 (2006) 236--277.

\bibitem{crespo2022new}
{\'A}.~Crespo-Blanco, L.~Gasi{\'n}ski, P.~Harjulehto, P.~Winkert,
A new class of double phase variable exponent problems: Existence and uniqueness, J. Differential Equations 323 (2022) 182--228.  

\bibitem{DeF18}
C. De Filippis,
Higher integrability for constrained minimizers of integral functionals with {$(p,q)$}-growth in low dimension,
Nonlinear Anal. { 170} (2018)  1--20.

\bibitem{Diening}
L.~Diening, P.~Harjulehto, P.~H\"{a}st\"{o}, M.~R$\mathring{\text{u}}$\v{z}i\v{c}ka,
Lebesgue and Sobolev Spaces with Variable Exponents,
Springer, Heidelberg, 2011.


\bibitem{ELM} L. Esposito, F. Leonetti, G. Mingione, Sharp regularity for functionals with $(p, q)$ growth, J. Differ. Equ. 204 (2004) 5--55.


\bibitem{fan2012}
X.~Fan,  An imbedding theorem for Musielak--Sobolev spaces, Nonlinear Anal. 75 (4) (2012) 1959--1971.





\bibitem{fanzhao2001}
X.~ Fan, D.~Zhao,
On the spaces $L^{p(x)}(Ω)$ and $W^{m,p(x)}(\Omega)$, J. Math. Anal. Appl. 263 (2001) 424–446.


\bibitem{Fig-Fur07} G.M. Figueiredo, M. Furtado
Positive solutions for some quasilinear equations with critical and supercritical growth Nonlinear Anal. 66 (7) (2007) 1600--1616.





\bibitem{Garcia} J. Garc\'{\i}a Azorero, I. Peral Alonso, Multiplicity of solutions for elliptic problems with critical exponent or with a nonsymmetric term, Trans. Amer. Math. Soc. 323 (2) (1991) 877--895.











\bibitem{HH25} H.H Ha, K. Ho, On critical double phase problems in $\mathbb{R}^N$ involving variable exponents, J. Math. Anal. Appl. 541 (2025) 128748.

\bibitem{HH.book}
P. Harjulehto, P. H\"ast\"o, Orlicz Spaces and Generalized Orlicz Spaces, Springer, Cham, 2019.

\bibitem{HHT17}
P. Harjulehto, P. H\"{a}st\"{o}, O. Toivanen,
H\"{o}lder regularity of quasiminimizers under generalized growth conditions, Calc. Var. Partial Differential Equations 56 (2) (2017) Paper No. 22, 26pp.



\bibitem{HL08}
C. He, G. Li, The regularity of weak solutions to nonlinear scalar field elliptic equations containing p\&q-Laplacians, Ann. Acad. Sci. Fenn. Math 33 (2) (2008) 337--371.



\bibitem{HKWZ}
K. Ho, Y.-H. Kim, P. Winkert, C. Zhang, The boundedness and H{\"o}lder continuity of weak solutions to elliptic equations involving variable exponents and critical growth, J. Differential Equations 313 (2022) 503--532.



\bibitem{Ho-Sim-2015}
K. Ho, I. Sim, Corrigendum to ``Existence and some properties of solutions for degenerate elliptic equations with exponent variable'' [Nonlinear Anal. 98 (2014) 146--164], 	Nonlinear Anal. 128 (2015) 423--426.

\bibitem{HW2022}
K. Ho,  P. Winkert,
New embedding results for double phase problems with variable exponents and a priori bounds for corresponding generalized double phase
problems, Calc. Var. Partial Differential Equations 62 (8) (2023) 227.







\bibitem{Li90} G. Li, Some properties of weak solutions of nonlinear scalar field equations,  Ann. Acad. Sci. Fenn. Ser. A I Math. 15 (1) (1990) 27--36.

\bibitem{LW13} G. Li, C. Wang, The existence of a nontrivial solution to  $p$-Laplacian equations in  $\R^N$  with supercritical growth,
Math. Methods Appl. Sci. 36 (1) (2013) 69--79.






\bibitem{Mar89b} P. Marcellini, 	Regularity of minimizers of integrals of the calculus of variations with nonstandard growth conditions, Arch. Rational Mech. Anal.  105 (3) (1989)  267--284.

\bibitem{Mar91} 	P. Marcellini, 	Regularity and existence of solutions of elliptic equations with {$p,q$}-growth conditions, J. Differential Equations 90 (1) (1991) 1--30.

\bibitem{Mus} J. Musielak, Orlicz Spaces and Modular Spaces, Springer-Verlag, Berlin, 1983.

\bibitem{Ok18}
J. Ok,
Partial regularity for general systems of double phase type with continuous coefficients,
Nonlinear Anal. 177 (2018) 673--698.

\bibitem{Ok20} J. Ok, Regularity for double phase problems under additional integrability assumptions, Nonlinear Anal. 194 (2020) 111408.




\bibitem{Rab74} P.H. Rabinowitz, Variational methods for nonlinear elliptic eigenvalue problems, Indiana Univ. Math. J. 23 (1974) 729--754.


\bibitem{RT20}
M.A. Ragusa, A. Tachikawa,
Regularity for minimizers for functionals of double phase with variable exponents, Adv. Nonlinear Anal. 9 (1) (2020) 710--728.

\bibitem{RK25} D. Ri, S. Kwon, Holder continuity and higher integrability of weak solutions to double phase elliptic equations involving variable exponents and critical growth, Nonlinear Anal. 255 (2025) 113754. 

\bibitem{Ru2000} M. R\accent23u\v{z}i\v{c}ka, Electrorheological Fluids: Modeling and Mathematical Theory, in: Lecture Notes in Mathematics, vol. 1748,
Springer-Verlag, Berlin, 2000.






\bibitem{Sun-Wu-2014} J. Sun, T.-F. Wu, Ground state solutions for an indefinite Kirchhoff type problem with steep potential well, J. Differential Equations 256 (2014) 1771–1792.



\bibitem{Zhao24} L. Zhao, On some multiple solutions for a $p(x)$-Laplace equation with supercritical growth, arXiv preprint (2024), arXiv:2409.10984.

\bibitem{Zhikov-1986} V. V. Zhikov, Averaging of functionals of the calculus of variations and
elasticity theory, Izv. Akad. Nauk SSSR Ser. Mat. 50 (1986) 675--710.

\bibitem{Zhikov-1995}
V.V. Zhikov,
On Lavrentiev's phenomenon,  Russian J. Math. Phys. 3 (1995) 249--269.

\bibitem{Zhikov1997}
V.V. Zhikov,
On some variational problems,  Russian J. Math. Phys. 5 (1997) 105--116. 

\bibitem{ZKO}
V.V. Zhikov, S.M. Kozlov,  O.A. Oleinik,
Homogenization of differential operators and integral functionals,  Springer, Berlin, 1994.

\end{thebibliography}
 \end{document}